\documentclass{amsart}%
\usepackage{amsmath}
\usepackage{amsfonts}
\usepackage{amssymb}
\usepackage{graphicx}%
\providecommand{\U}[1]{\protect\rule{.1in}{.1in}}
\newtheorem{theorem}{Theorem}

\newtheorem{conjecture}[theorem]{Conjecture}
\newtheorem{corollary}[theorem]{Corollary}

\newtheorem{definition}[theorem]{Definition}

\newtheorem{lemma}[theorem]{Lemma}

\newtheorem{problem}[theorem]{Problem}

\newtheorem{proposition}[theorem]{Proposition}
\newtheorem{remark}[theorem]{Remark}

\begin{document}

\title{Rank, combinatorial cost and homology torsion growth in higher rank lattices}
\author{Miklos Abert, Tsachik Gelander and Nikolay Nikolov}

\thanks{M. A. is supported by the ERC Consolidator Grant 648017 and the MTA Lendulet Groups and Graph limits grant,
T.G. acknowledges the support of ISF-Moked grant 2095/15,
N.N. acknowledges the support of EPSRC grant EP/H045112/1 and the Clay Mathematical Institute}

\date{}

\begin{abstract}
We investigate the rank gradient and growth of torsion in homology in
residually finite groups. As a tool, we introduce a new complexity notion
for generating sets, using measured groupoids and combinatorial cost. 

As an application we prove the vanishing of the above invariants for Farber
sequences of subgroups of right angled groups. A group is right angled if it
can be generated by a sequence of elements of infinite order such that any
two consecutive elements commute. 

Most non-uniform lattices in higher rank simple Lie groups are right angled. We
provide the first examples of uniform (co-compact) right angled arithmetic
groups in $\mathrm{SL}(n,\mathbb{R}),~n\geq 3$ and $\mathrm{SO}(p,q)$ for
some values of $p,q$. This is a class of lattices for which the Congruence
Subgroup Property is not known in general.

Using rigidity theory and the notion of invariant random subgroups it
follows that both the rank gradient and the homology torsion growth vanish
for an arbitrary sequence of subgroups in any right angled lattice in a
higher rank simple Lie group.
\end{abstract}

\maketitle

\section{Introduction}
In this paper we deal with the asymptotic behavior of two delicate group invariants, the minimal size of a generating set (the rank), and the size of the torsion part of the first homology. In general, it is extremely hard to determine the rank, or even decide whether it is finite.
Finite generation of lattices in connected semisimple Lie groups $G$ was proved case by case by many different authors, notably in the works of Garland and Raghunathan \cite{Ga-Ra} for $\mathbb R$-rank$(G) = 1$ and of Kazhdan \cite{Kazhdan} for the case where all the factors of $G$ have $\mathbb R$-rank $\ge 2$.  
Torsion in homology is a more recent object of investigation, related to algebraic topology and number theory (see \cite{luck,berg}). 

The asymptotic behavior of Betti numbers and more generally, representation multiplicities associated to lattices in Lie groups have been extensively studied since the 80's (see \cite{degeorge, delorme} and the introduction of \cite{Samurais}). In recent years, there has been new progress in this direction (see \cite{Samurais,samurai,LaFi1,LaFi2, FLM}). These invariants can be effectively handled using the analytic side of the theory and limiting arguments, like the L\"{u}ck approximation theorem \cite{luck}. For the rank and the torsion in homology, there do exist some analytic tools, namely the cost and the $L^2$-torsion. However, their efficiency is much more limited. 
 
In this paper we develop new, geometric tools to estimate the rank and the torsion in homology for subgroups of finite index in a special class of groups. This class includes non-uniform as well as uniform lattices in simple Lie groups. 

\medskip 

\subsection{Rank and homology torsion growth in right angled groups}
For a finitely generated group $\Gamma$ let $d(\Gamma)$ denote the minimal
number of generators (or rank) of $\Gamma$. For a subgroup $H\leq\Gamma$ of
finite index let
\[
r(\Gamma,H)=(d(H)-1)/\left\vert \Gamma:H\right\vert.
\]
The \emph{rank gradient} of $\Gamma$ with respect to a sequence
$(\Gamma_{n})$ of finite index subgroups is defined to be
\[
\mathrm{RG}(\Gamma,(\Gamma_{n}))=\lim_{n\rightarrow\infty}r(\Gamma,\Gamma
_{n})
\]
when this limit exists. This notion has been introduced by Lackenby
\cite{lack} and further investigated in the literature for \emph{chains} of
subgroups. Recall that a chain in $\Gamma$ is a decreasing sequence
$\Gamma=\Gamma_{0}>\Gamma_{1}>\ldots$ of subgroups of finite index in $\Gamma
$. In this case, it is easy to see that $r(\Gamma,\Gamma_{n})$ is
non-increasing and so the limit exists.

For a finitely generated group $H$ let $\mathrm{trs}(H)$ denote the size of
the torsion part of the Abelianization (or first homology group) $H/H^{\prime
}$. The \emph{homology torsion growth} of $\Gamma$ with respect to
$(\Gamma_{n})$ is defined as
\[
\text{Trs}(\Gamma,(\Gamma_{n}))=\lim_{n\rightarrow\infty}\frac{\ln\mathrm{trs}%
(\Gamma_{n})}{\left\vert \Gamma:\Gamma_{n}\right\vert }
\]
when this limit exists. It is easy to see that when $\Gamma$ is finitely presented this sequence is bounded, see Lemma \ref{simplebound}. The homology torsion growth has been analysed in \cite{berg} and \cite{Lueck}.
In this paper we deal only with the first homology.

The following definition is the central object of investigation of this paper. 

\begin{definition}
The group $\Gamma$ is \emph{right-angled} if it admits a finite generating list
$\{\gamma_{1},\ldots,\gamma_{m}\}$ of elements of infinite order such that
$[\gamma_{i},\gamma_{i+1}]=1$ for $i=1,\ldots,m-1$.
\end{definition}

So, a right-angled group is a quotient of a right-angled Artin group with a connected defining graph where the order of the generators stay infinite. The importance of this class is reflected both in its ubiquity and in the strength of the results one can prove for such groups, mainly when combined together with rigidity theory. There are plenty of non-uniform and, as we shall prove, some uniform lattices in higher rank simple Lie groups that are right angled.  

We now state our main vanishing theorems concerning rank and homology torsion growth for higher rank right angled lattices.

\begin{theorem}
\label{main1}
Let $G$ be a simple real Lie group of real rank at least $2$ and let $\Gamma$ be
a right-angled lattice in $G$. Then for any sequence of subgroups $(\Gamma
_{n})$ in $\Gamma$ such that $\left\vert \Gamma:\Gamma_{n}\right\vert
\rightarrow\infty$ we have
\[
\mathrm{RG}(\Gamma,(\Gamma_{n}))=0
\text{.}
\]
\end{theorem}

One can attempt to further relax on the condition of having a common ambient
lattice $\Gamma$ in Theorem \ref{main1} and just consider arbitrary sequences
of lattices in the fixed Lie group $G$. It was shown in \cite{Ge} that given a
simple Lie group $G$ and fixing the Haar measure on $G$, there exists a
constant $C$ such that
\[
d(\Gamma)\leq C\cdot\text{Vol}(G/\Gamma)
\]
for every lattice $\Gamma\leq G$. In higher rank, the growth of rank seems to
be sub-linear in the volume. We make the following provocative conjecture.

\begin{conjecture}
\label{kerdes}Let $G$ be a higher rank simple Lie group with a fixed Haar
measure. Let $\Gamma_{n}$ be a sequence of lattices in $G$ with $\rm{Vol}(G/\Gamma
_{n})\rightarrow\infty$. Then
\[
\lim_{n\rightarrow\infty}\frac{d(\Gamma_{n})-1}{\rm{Vol}(G/\Gamma_{n}%
)}=0\text{.}
\]

\end{conjecture}

In rank $1$, this is false in general. Indeed, if $G$ admits a lattice
$\Gamma$ that surjects on a non-abelian free group $F$, then the set of finite
index subgroups in $F$ leads to a counterexample. (For instance, the groups
$SO(n,1)$ ($n\geq2$) admit such co-compact as well as non-co-compact lattices). However, our conjecture
does admit a natural generalization to arbitrary Lie groups, using the
framework of invariant random subgroups. See Section \ref{sec:open} for details.

Note that Theorem \ref{main1} together with the existence results (Theorem \ref{coco} and Theorem \ref{SO(7,2)}) below provide the first examples
of sequences of cocompact lattices in a simple Lie group of higher rank where
the rank grows sublinearly in the volume.

We now turn to the main result concerning torsion.

\begin{theorem}\label{main2}
Let $G$, $\Gamma$ and $(\Gamma_n)$ be as in Theorem \ref{main1}. Then
\[
\rm{Trs}(\Gamma,(\Gamma_{n}))=0.
\]
\end{theorem}

Note that in this situation, $\Gamma_{n}$ has property (T)\ and so
$\mathrm{trs}(\Gamma_{n})$ equals the size of the Abelianization of
$\Gamma_{n}$. 

\medskip 

For \emph{non-uniform} lattices, the analogues of both Theorem \ref{main1} and Theorem \ref{main2} can be derived
from the existing literature, using Raghunathan's theorem \cite{Ra1,Ra2} that such lattices satisfy the so-called
Congruence Subgroup Property (CSP), adding a result of Sharma and
Venkataramana \cite{venky} and some finite group theory. We will do this in
Section \ref{estimates} where we also give effective bounds on both the torsion and the
rank when the CSP\ holds. However, the CSP is far from being proved in the
cocompact (uniform) case and so the above, arithmetic methods fail. Our proofs for Theorems \ref{main1} and \ref{main2} are inherently
geometric and do not use arithmetic arguments.

\subsection{Existence of cocompact right angled lattices}
We prove the existence of co-compact higher
rank right angled lattices in various classes of simple Lie groups. In particular we obtain:

\begin{theorem}
\label{coco}The groups $SL_{n}(\mathbb{R})$ admit right angled co-compact lattices for
$n>2$.
\end{theorem}

For a detailed description of these examples,  see Section \ref{section:SL(n)}. In particular, it is not known whether these lattices satisfy the CSP (or bounded generation).
In addition, we demonstrate our method for constructing right angled lattices in some special orthogonal groups where the geometric picture is more transparent:

%
%

\begin{theorem}\label{SO(7,2)}
\label{O} Let $K=\mathbb{Q}(\sqrt{2})$ and denote its ring of integers
$\mathcal{O}_{K}=\mathbb{Z}[\sqrt{2}]$. For an integer  $n \geq 7$ let
\[
f(x_{1},\ldots,x_{n+2})=x_{1}^{2}+\ldots+x_{n}^{2}-\sqrt{2}x_{n+1}^{2}-\sqrt
{2}x_{n+2}^{2}%
\]
be a quadratic form on $V=K^{n+2}$. Let $\mathbb{G}=SO(f)$ and $\Gamma
=\mathbb{G}(\mathcal{O}_{K})$. Then $\Gamma$ is a uniform lattice in
$G=\mathbb{G}(\mathbb{R})\cong SO(n,2)$ admitting a finite index right-angled subgroup.
\end{theorem}

It is natural to ask at this point which simple Lie groups admit right-angled
cocompact lattices. It is clear that by our method such lattices can be
constructed in every orthogonal group $SO(p,q)$ as long as $q\geq2$ and $p$ is
not too small. 

These new examples are interesting in their own right, as they seem to possess properties that usually reflect higher $\mathbb{Q}$-rank, i.e. properties of non-uniform arithmetic groups. Thus they might be suitable candidates to study various notions that are known to hold for non-uniform lattices but are widely open for uniform lattices,
such as bounded generation.

\subsection{Combinatorial cost and sofic approximations}
The proofs of Theorem \ref{main1} and Theorem \ref{main2} use combinatorial cost, a notion introduced
by Elek \cite{elek}. This is a metric invariant of a sequence of finite graphs
$G_{n}$. A \emph{rewiring} of $G_{n}$ is another sequence of graphs $H_{n}$ on
the same vertex set as of $G_{n}$, such that the bi-Lipshitz distortion of the
identity map stays bounded in $n$. The combinatorial cost $cc(G_{n})$ is the
infimum of the edge densities of possible rewirings of $G_{n}$.

A theorem of Gaboriau \cite{gabor} says that every free action of a right
angled group has cost $1$. Adapting his proof to the finite graph setting
gives us the following effective version.

\begin{theorem}
\label{ccsofapprox}Let $\Gamma$ be a right angled group and let $(G_{n})$ be a
sofic approximation of $\Gamma$. Then there exists $\varepsilon_{k}%
\rightarrow0$ and rewirings of $(G_{n})$ of edge density $1+\varepsilon_{k}$
with bi-Lipshitz distortion that is polynomial in $1/\varepsilon_{k}$. In
particular, the combinatorial cost of $(G_{n})$ equals $1$.
\end{theorem}

For the notion of a sofic approximation see Section \ref{soficap}. We shall say that a graph
sequence satisfying the conclusion of Theorem \ref{ccsofapprox} has \emph{polynomial distortion}. One can define subexponential distortion of graph sequences in a similar fashion.

Let $G$ be a locally compact group. Then for any closed subgroup $H$ of $G$
where the normalizer of $H$ in $G$ has finite covolume, one can define the
invariant random subgroup $\mu_{H}$ by taking a uniform random conjugate of
$H$ (against the Haar measure). A sequence of subgroups $\Gamma_{n}\leq G$ is
\emph{Farber}, if $\mu_{\Gamma_{n}}$ weakly converges to $\mu_{1}$. When $G$ is a
discrete group generated by a finite symmetric set $S$, then $\Gamma_{n}$ is
Farber, if and only if the Schreier graphs $\mathrm{Sch}(G/\Gamma_{n},S)$ give
a sofic approximation of $G$.

Our next theorem connects the rank gradient of a Farber sequence to the
combinatorial cost of the corresponding sofic approximation.

\begin{theorem}\label{rgcc}
Let $\Gamma$ be a group generated by the finite symmetric set $S $
containing the identity and let $(\Gamma_{n})$ be a Farber sequence of
subgroups of finite index in $\Gamma$ such that $\mathrm{RG}(\Gamma,(\Gamma_{n}))$ exists. Then
\[
\mathrm{RG}(\Gamma,(\Gamma_{n}))\leq\mathrm{cc}(\mathrm{Sch}(\Gamma,\Gamma
_{n},S))-1\text{.}
\]
If $(\Gamma_{n})$ is a chain, then there is equality above.
\end{theorem}

The proof goes through interpreting the rank of $\Gamma_{n}$ as the minimal
number of generators for the measured groupoid associated to the action of
$\Gamma$ on $\Gamma/\Gamma_{n}$. This approach was initiated in \cite{miknik}.

The next result controls the homology torsion for sofic approximations where
there are effective rewirings to get close to the combinatorial cost.

\begin{theorem}\label{tortheorem}
\label{subexptorsion}Let $\Gamma$ be a finitely presented group. Let
$(\Gamma_{n})$ be a Farber sequence of subgroups in $\Gamma$ such that the
corresponding sofic approximation of $\Gamma$ has subexponential distortion.
Then the homology torsion growth
\[
\rm{Trs}(\Gamma,(\Gamma_{n}))=0.
\]

\end{theorem}
Looking at the rank of a subgroup through
combinatorial cost (using the cost of the measured groupoid associated to an
action) seems to be the right way to define generating sets of the stabilizer
of `low complexity'. We expect more applications of this point of view.
\bigskip


\subsection{Rigidity}
The final ingredient of the proof of Theorem \ref{main1} as well as of Theorem \ref{main2} comes from
\cite{Samurais}. The following is implicitly proved there, using the notion
of invariant random subgroups.

\begin{theorem}\label{thm:samurais}
Let $\Gamma$ be a lattice in a higher rank simple Lie group. Then any sequence
$\Gamma_{n}$ of finite index subgroups with $\left\vert \Gamma:\Gamma
_{n}\right\vert \rightarrow\infty$ is Farber.
\end{theorem}

The vanishing of the rank gradient, Theorem \ref{main1}, now follows from the combination of Theorem \ref{thm:samurais}
together with the soft form of Theorem \ref{ccsofapprox} and Theorem
\ref{rgcc}. The vanishing of the homology torsion growth, Theorem \ref{main2}, follows from Theorem \ref{thm:samurais} together with the
effective form of Theorem \ref{ccsofapprox}, which gives polynomial
distortion, together with Theorem \ref{subexptorsion}. 

\subsection{Algebraic groups over non-Archimedean local fields}
Our methods and ideas can also be applied to lattices in higher rank algebraic groups over local fields. We demonstrate this by proving the following:

\begin{theorem} \label{p} Let $\mathbb F$ be a finite field and let $\Gamma= SL(n, \mathbb F[t])$ where $n>2$ and $\mathbb F[t]$ is the ring of polynomials over $\mathbb F$. Then
$\lim_{i \rightarrow \infty} r(\Gamma, \Gamma_i)=0$ for every sequence of subgroups $\Gamma_i$ of $\Gamma$ with  $|\Gamma : \Gamma_i| \rightarrow \infty$.
\end{theorem} 

What makes it possible to apply the rigidity part of the argument is the recent extension of Stuck--Zimmer theorem for the context that was proved by A. Levit \cite{Arie} and the analysis of invariant random subgroups in non-archimedean groups that was carried out in \cite{GL}.

\bigskip

The structure of the paper is as follows. In Section \ref{irs} we recall the notion of invariant random subgroups from \cite{Samurais} and deduce Theorem \ref{thm:samurais}. Section \ref{soficap} covers defintion and basic properties of sofic aproximations. The notion of combinatorial cost  and its relation to rank gradient is discussed in Section
\ref{Section:CombCost} where we prove Theorem \ref{rgcc}. Section \ref{torsion} applies these results to torsion homology growth and proves Theorem \ref{tortheorem}. We also obtain some explicit bounds on the growth in case when $\Gamma$ is an arithmetic group. In Section \ref{sec:coco} we construct the right angled lattices, proving Theorems \ref{coco} and \ref{SO(7,2)}. Section \ref{sec:local} discusses the case of lattices over non-archimedean local fields and proves Theorem \ref{p}. In the final Section \ref{sec:open} we list open problems and suggest further directions for investigation motivated by our results.

\section{The role of IRS and rigidity.} \label{irs}

Recall that an IRS on $\Gamma$ is a conjugacy invariant probability measure on the space $\text{Sub}_\Gamma$ of subgroups of $\Gamma$ equipped with the Chabauty topology. See \cite{abglvi} where the notion was introduced. 
We denote by $\rm{IRS}(\Gamma)$ the set of all IRS on $\Gamma$. This is a compact space under the weak$^*$ topology.
Every subgroup $\Delta$ whose normalizer is of finite index in $\Gamma$ induces an IRS, denoted $\mu_\Delta$, which is the uniform measure on the conjugacy class of $\Delta$. 

We say that a chain $(\Gamma_{n})$ in $\Gamma$ is \emph{Farber}, if the action
of $\Gamma$ on the boundary of its coset tree $T=T(\Gamma,(\Gamma_{n}))$ is
essentially free, that is, if almost every element of $\partial T$ has trivial
stabilizer in $\Gamma$. This is the case e.g., when the chain is normal with
trivial intersection. More generally:

\begin{lemma}
Let $(\Gamma_{n})$ be a chain of finite index subgroups of $\Gamma$. Then $(\Gamma_{n})$ is Farber iff 
$\mu_{\Gamma_n}\to\mu_{1_\Gamma}$ in the weak star topology on $\text{Prob}(\text{Sub}_\Gamma)$. 
\end{lemma}
Thus the notion of IRS $\mu_n$ converging to the trivial IRS $\mu_{1_\Gamma}$ is a generalization of a Farber chain.

\begin{proof} Since $\Gamma$ is countable the chain $(\Gamma_n)$ is Farber if and only if any $g \in \Gamma \backslash \{1\}$ fixes a set of measure zero in $\partial T(\Gamma, (\Gamma_n))$ i.e. \[  \frac{| \{ \Gamma_na \ | \ a\in \Gamma, \ g \in (
\Gamma_n)^a \}|}{|\Gamma :\Gamma_n|} \rightarrow 0 \textrm{ with }n \rightarrow \infty. \] We can interpret the fraction on the left hand side as the probability that the element $g$ lies in a $\mu_{\Gamma_n}$-random subgroup from $\rm{Sub}_\Gamma$. Again countability of $\Gamma$ gives that this condition is equivalent to $\mu_{\Gamma_n} \rightarrow \mu_{1_\Gamma}$ in the weak$^*$ topology.
\end{proof}

\bigskip

The following result is implicit in \cite{Samurais}.

\begin{theorem}\label{mu_n}
Let $G$ be a higher rank simple Lie group and $\Gamma\le G$ a lattice. Let $\Gamma_n\le \Gamma$ be a sequence of finite index subgroups with $|\Gamma:\Gamma_n|\to\infty$. Then the sequence of IRS's $\mu_{\Gamma_n}$ converges to $\mu_{1_\Gamma}$ in $\text{Prob}(\text{Sub}_\Gamma)$.
\end{theorem}

More precisely by \cite[Theorem 4.2]{Samurais} the IRS $\mu_{\Gamma_n}^G$ in $G$  corresponding to the lattices $\Gamma_n\le G$ converge to $\mu_{1_G}$ as measures on $\text{Sub}_G$. However $\mu_{\Gamma_n}^G$ are induced from $\mu_{\Gamma_n}$ (see \cite{Samurais} for details about inducing IRS from a lattice to the ambient group). Hence Theorem \ref{mu_n} is a consequence of:

\begin{lemma} \label{induced}
Let $G$ and $\Gamma\le G$ be as above. Let $\mu_n$ be IRSs on $\Gamma$ and denote by $\mu_n^G=\text{Ind}_\Gamma^G(\mu_n)$ the induced IRSs on $G$. Then $\mu_n^G$ converges to  $\mu_{1_G}$ if and only if $\mu_n$ converges to $\mu_{1_\Gamma}$.
\end{lemma}

\begin{proof}
For an element $\gamma\in \Gamma$ denote by $m_n(\gamma)$ the $\mu_n$ probability that $\gamma$ belongs to a random subgroup, i.e.
$$
 m_n(\gamma)=\mu_n(\{\Delta\in\text{Sub}_\Gamma:\gamma\in\Delta\}).
$$
Since $\Gamma$ is countable we have that $\mu_n\to\mu_{1_\Gamma}$ iff $m_n(\gamma)\to 0$ for every $\gamma\in\Gamma\setminus\{1\}$. Moreover, since every IRS in $G$ (as well as in $\Gamma$) without an atom on the trivial subgroup is supported on Zariski dense subgroups (see \cite[Theorem 2.6]{Samurais}) we have that $\mu_n\to\mu_{1_\Gamma}$ iff $m_n(\gamma)\to 0$ for every semisimple element $\gamma\in\Gamma\setminus\{1\}$.

Suppose that this is not the case and let $\gamma\in\Gamma\setminus\{1\}$ be a semisimple element such that $m_n(\gamma)$ does not tend to $0$. Since $\gamma$ is semisimple, its conjugacy class $\gamma^G$ is closed in $G$ and hence the set 
$$
 \Theta(\gamma):=\{ H\in \text{Sub}_G: \gamma^G\cap H\ne\emptyset\}
$$
is compact in $\text{Sub}_G$. Let $\psi$ be a continuous non-negative function on $\text{Sub}_G$ equals to $1$ on $\Theta(\gamma)$ and vanishes on $1_G$. Then $\int\psi d{\mu_n^G}\ge m_n(\gamma)$ and hence does not tend to 
$$
 0=\psi(\langle 1_G\rangle)=\int\psi d{\mu_{1_G}},
$$ 
implying that $\mu_n^G$ does not converge to $\mu_{1_G}$.

Let us now prove the converse.  Let $\Omega$ be a fundamental domain for $\Gamma$ in $G$ and suppose that the Haar measure of $G$ is normalized so that $\Omega$ has measure $1$. The definition of induced measure can be expressed by an integral as follows:
$$
 \mu_n^G(f)=\int_{G/\Gamma}\mu_n(f^g|_{\text{Sub}_\Gamma})d(g\Gamma)=\int_{\Omega}\big(\int_{\text{Sub}_\Gamma} f^w(\Delta) d\mu_n(\Delta)\big)dw,
$$  
where $f$ is an arbitrary continuous function on $\text{Sub}_G$ and 
$$
 f^w(H):=f(w^{-1}Hw)~\text{for}~ w\in G,~H\in\text{Sub}_G.
$$ 
Note that although $f^g$ is not well defined for $g\in G/\Gamma$, the quantity $\mu_n(f^g|_{\text{Sub}_\Gamma})$ is well defined since $\mu_n$ is $\Gamma$-conjugacy invariant.

Suppose that $\mu_n^G$ does not converge to $\mu_{1_G}$. Then 
there is a continuous function $f:\text{Sub}_G\to\mathbb{R}^{\ge 0}$ such that $f(\langle 1\rangle)=0$ and $\mu_n^G(f)\nrightarrow 0$. By the dominated convergence theorem it follows that there is a set of positive measure $U\subset \Omega$ such that $\int_{\text{Sub}_\Gamma} f^u(\Delta)d\mu_n(\Delta)$ does not converge to $0$ for every $u\in U$. Picking $u_0\in U$, and regarding $f^{u_0}$ as a function on the subspace $\text{Sub}_\Gamma\subset\text{Sub}_G$ we have that 
$$
\mu_n(f^{u_0})=
 \int_{\text{Sub}_\Gamma} f^{u_0}(\Delta)d\mu_n(\Delta)\nrightarrow 0=f(\langle 1\rangle)=f^{u_0}(\langle 1\rangle)=\mu_{1_{\Gamma}}(f^{u_0}),
$$
implying that $\mu_n\nrightarrow\mu_{1_{\Gamma}}$.

\end{proof}


\section{Sofic approximations} \label{soficap}

A countable group $\Gamma$ is \emph{sofic}, if for every finite $B\subseteq
\Gamma$ and $\varepsilon>0$ there exists a finite set $V$ and a map
$\sigma:\Gamma\rightarrow\mathrm{Sym}(V)$ such that
\begin{equation}
\left\vert \left\{  v\in V\mid v\sigma(a)\neq v,v\sigma(a)\sigma
(b)=v\sigma(ab)\text{ (}a,b\in B,a\neq e\text{)}\right\}  \right\vert
>(1-\varepsilon)\left\vert V\right\vert \tag{Sof}%
\end{equation}
That is, if the pair $(V,\sigma)$ looks like a free action of $\Gamma$ from
most $v\in V$.

The notion was invented by Gromov \cite{gromov} and further clarified by Weiss \cite{weiss} as a common generalization of
amenable groups and residually finite groups. As of now, there is no countable group that is known to be non-sofic. 

A \emph{sofic approximation of }$\Gamma$ is a sequence of pairs $(V_{n}%
,\sigma_{n})$ such that for every finite subset $B$ of $\Gamma$ there exists
$\varepsilon_{n}\rightarrow0$ \ such that (Sof) holds for $V=V_{n}$,
$\sigma=\sigma_{n}$ and $\varepsilon=\varepsilon_{n}$ ($n\geq1$). 

When $\Gamma$ is generated by a finite set $S$, one can visualize a sofic
approximation by drawing the Schreier graphs $\mathrm{Sch}(V_{n},S,\sigma
_{n})$ as follows. The vertex set equals $V_{n}$ and for all $x\in V_{n}$ and
$s\in S$ there is an $s$-labeled edge going from $x$ to $x\sigma_{n}(s)$. Let
$\mathrm{Cay}(\Gamma,S)$ denote the Cayley graph of $\Gamma$ with respect to
$S$. In a graph $G$ let $B_{R}(G,v)$ denote the $R$-ball centered at the
vertex $v$. For two labelled graphs $G$ and $G'$ and vertices $v \in G, v' \in G'$ we write
$B_{R}(G,v)\cong B_{R}(G',v')$ when $B_R(G,v)$ and $B_R(G',v')$ are isomorphic as labelled graphs rooted at $v$ and $v'$ respectively. Applying (Sof) gives the following. 

\begin{lemma}
\label{graf}Let $\Gamma$ be generated by the finite symmetric set and let
$(V_{n},\sigma_{n})$ be a sofic approximation for $\Gamma$. Let $G_{n}%
=\mathrm{Sch}(V_{n},S,\sigma_{n})$ and let $G=\mathrm{Cay}(\Gamma,S)$. Then
for every $R>0$, we have
\[
\lim_{n\rightarrow\infty}\mathcal{P}(B_{R}(G_{n},v)\cong B_{R}(G,e))=1\text{.}
\]
where $v \in V(G_{n})$ is uniform random. 
\end{lemma}

Clearly, every sequence of $S$-labeled graphs as above defines a sofic
approximation. In other words, a sequence of $S$-labeled graphs is realized by
a sofic approximation if and only if it Benjamini-Schramm converges to the
Cayley graph $\mathrm{Cay}(\Gamma,S)$. 

Recall that if  $\Gamma$ is a group and $H\leq\Gamma$ is such that the normalizer
$N_{\Gamma}(H)$ has finite index in $\Gamma$, then $\mu_{H}$ denotes the uniform
probability measure on the set of $\Gamma$-conjugates of $H$.

For a finite set of symbols $S$ let $F_{S}$ denote the free group with
alphabet $S$. In the following, we give an alternate definition of soficity
using invariant random subgroups.

\begin{lemma}
\label{ujirs}Let $\Gamma$ be a finitely generated group. Then $\Gamma$ is
sofic, if and only if there exists a finitely generated free group $F$, a
normal subgroup $N\vartriangleleft F$ with $\Gamma\simeq F/N$ and a
sequence of subgroups $H_{n}\leq F$ of finite index such that
\[
\lim_{n\rightarrow\infty}\mu_{H_{n}}=\mu_{N}\text{.}%
\]

\end{lemma}

\begin{proof}Assume $\Gamma$ is sofic. Let $S$ be a finite
symmetric generating set of $\Gamma$. Let $\Phi:F_{S}\rightarrow\Gamma$ be the
canonical homomorphism with kernel $N=\ker\Phi$. Let $(V_{n},\sigma_{n})$ be a
sofic approximation of $\Gamma$. By changing the value of $\sigma_{n}$ on a
subset of $V_n$ of proportion which tends to zero with $n$, one can assume that $\sigma_{n}(s^{-1})=\sigma_{n}(s)^{-1}$,
($s\in S,$ $n\geq1$). Then the maps $\sigma_{n}:S\rightarrow\mathrm{Sym}%
(V_{n})$ extend to homomorphisms $\Psi_{n}:F_{S}\rightarrow\mathrm{Sym}%
(V_{n})$. Again, by changing $\sigma_{n}$ on a subset of $V_n$ of small proportion, one can assume
that these actions are transitive. Let $v_{n}\in V_{n}$ and let $H_{n}$ be the
stabilizer of $v_{n}$ with respect to the action defined by $\Psi_{n}$. Let
$K_{n}$ be a uniform random conjugate of $H_{n}$, that is, a random subgroup
of $F_{S}$ with distribution $\mu_{H_{n}}$. Then $K_{n}$ is the stabilizer of
a uniform random element of $V_{n}$. If $g\in N$, then $g\in\ker\Psi_{n}$, so
\[
\mathcal{P}(g\in K_{n})=1\text{ \ (}n\geq1\text{).}%
\]
If $g\notin N$, then $\Phi(g)\neq1$. Let $n_{0}>0$ such that $\Phi(g)\in
S^{n_{0}}$. Then by (Sof), we have
\[
\mathcal{P}(g\in K_{n})=\frac{\left\vert \left\{  x\in V_{n}\mid x\sigma
_{n}(\Phi(g))=x\right\}  \right\vert }{\left\vert V_{n}\right\vert }%
\leq\varepsilon_{n}\text{ \ (}n\geq n_{0}\text{).}%
\]
Since $\varepsilon_{n}\rightarrow0$, we get that $\mathcal{P}(g\in
K_{n})\rightarrow0$. This proves that $\mu_{H_{n}}$ weakly converges to
$\mu_{N}$.

Assume now that there exists a free group $F$ with finite free generating set
$S$, a normal subgroup $N\vartriangleleft F$ with $\Gamma\simeq F/N$ and a
sequence of subgroups $H_{n}\leq F$ of finite index such that $\lim
_{n\rightarrow\infty}\mu_{H_{n}}=\mu_{N}$. Let $K_{n}$ be a uniform random
conjugate of $H_{n}$. Let $\Phi:F\rightarrow\Gamma$ be the quotient map by
$N$. For $g\in\Gamma$ choose $\phi(g)\in\Phi^{-1}(g)$. 

Let $V_{n}=F/H_{n}$ ($n\geq1$). Then $F$ acts on $V_{n}$ by the right coset
action. We define $\sigma_{n}:\Gamma\rightarrow\mathrm{Sym}(V_{n})$ as
\[
x\sigma_{n}(g)=x\phi(g)\text{ (}x\in V_{n}\text{).}%
\]
We claim that $(V_{n},\sigma_{n})$ is a sofic approximation for $\Gamma$. Let
$B\subseteq\Gamma$ be a finite subset. For $e\neq a\in B$ we have
$\phi(a)\notin N$, so by weak convergence, we have
\[
\lim_{n\rightarrow\infty}\frac{\left\vert \left\{  x\in V_{n}\mid x\sigma
_{n}(a)=x\right\}  \right\vert }{\left\vert V_{n}\right\vert }=\lim
_{n\rightarrow\infty}\mathcal{P}(\phi(a)\in K_{n})=0\text{.}%
\]

For $a,b\in B$ we have $w=\phi(a)\phi(b)\phi(ab)^{-1}\in N$, so by weak
convergence, we have
\[
\lim_{n\rightarrow\infty}\frac{\left\vert \left\{  x\in V_{n}\mid x\phi
(a)\phi(b)=x\phi(ab)\right\}  \right\vert }{\left\vert V_{n}\right\vert }%
=\lim_{n\rightarrow\infty}\mathcal{P}(w\in K_{n})=1\text{.}%
\]
The claim holds. \end{proof} \bigskip

\section{Combinatorial cost and rank gradient}\label{Section:CombCost}
For a graph $G$ let $d_{G}$ denote the graph metric on $V(G)$. For $G=(V,E)$
and $G^{\prime}=(V,E^{\prime})$ we define the \emph{bi-Lipschitz distance of
}$G$\emph{ and }$G^{\prime}$ as
\[
d_{L}(G,G^{\prime})=\max\left\{  \max_{(x,y)\in E}d_{G^{\prime}}%
(x,y),\max_{(x,y)\in E^{\prime}}d_{G}(x,y)\right\}  \text{.}%
\]

Note that $d_{L}$ equals the bi-Lipschitz distortion of the identity map
$V\rightarrow V$ between the metric spaces $(V,d_{G})$ and $(V,d_{G^{\prime}%
})$.

Now we define the combinatorial cost of a graph sequence. Let $G_{n}%
=(V_{n},E_{n})$ be a sequence of finite graphs.

\begin{definition}
We say that the graph sequence $G_{n}^{\prime}=(V_{n},E_{n}^{\prime})$ is a
\emph{rewiring} of $G_{n}$ (denoted by $(G_{n})\sim(G_{n}^{\prime})$) if
$d_{L}(G_{n},G_{n}^{\prime})$ is bounded in $n$.
\end{definition}

Let the \emph{edge measure }of the graph sequence ($G_{n}$) be
\[
e(G_{n})=\lim\inf_{n\rightarrow\infty}\frac{\left\vert E_{n}\right\vert
}{\left\vert V_{n}\right\vert }\text{.}%
\]

\begin{definition}
Let the combinatorial cost of ($G_{n}$) be
\[
\mathrm{cc}(G_{n})=\inf_{(G_{n}^{\prime})\sim(G_{n})}e(G_{n}^{\prime})\text{.}%
\]
\end{definition}

\noindent
{\bf Remark:}
These notions were introduced by Elek in \cite{elek}. Note that for the purpose of the current paper it would be slightly more natural to define the edge measure as a $\lim\sup$. However, we decided not to change the original definition of Elek. As a result at some points we are led to make extra assumptions, for instance in Theorem \ref{rgcc} we have to assume that the rank gradient exists. See also Remark \ref{rem:limsup}.

\medskip

We now prove a vanishing theorem on the combinatorial cost of sofic
approximations of right angled groups. In the proof we adapt Gaboriau's method
\cite{gabor} who used it to prove that higher rank semisimple non-uniform
lattices have fixed price $1$.

\begin{theorem}
\label{rightangcc}Let $\Gamma$ be a right angled group and let $(V_{n}%
,\sigma_{n})$ be a sofic approximation of $\Gamma$. Then
\[
\mathrm{cc}(\mathrm{Sch}(V_{n},S,\sigma_{n}))=1
\]
for any finite symmetric generating set $S$ of $\Gamma$.
\end{theorem}

\begin{proof} Changing the generating set is a rewiring on the
Schreier graph level, so we can assume that $S$ is a right angled generating
list. That is,
\[
S=\left\{  s_{1},\ldots,s_{k}\right\}  \cup\left\{  s_{1}^{-1},\ldots
,s_{k}^{-1}\right\}
\]
where all $s_{i}$ have infinite order and $[s_{i},s_{i+1}]=1$ ($1\leq i<k$).
Let
\[
G_{n}=\mathrm{Sch}(V_{n},S,\sigma_{n})\text{ and }G=\mathrm{Cay}%
(\Gamma,S)\text{ \ (}n\geq1\text{).}%
\]
Note that $\left\{  s_{1},\ldots,s_{k}\right\}  $ may be a multiset. In that
case, we take the corresponding edges of $G_{n}$ and $G$ with multiplicity. We
will use the shortcut
\[
xg=x\sigma_{n}(g)\text{ \ (}g\in\Gamma,x\in V_{n},1\leq i<k,n\geq1\text{).}%
\]

Let $R>0$ be an even integer. We define a subgraph $H_{n}$ of $G_{n}$ as
follows. Let $C_{n,i}$ be the subgraph spanned by the $s_{i}$-edges in $G_n$ ($1\le  i\le k)$. The graph
$C_{n,i}$ is a disjoint union of cycles. Let $X_{n,1}=V(G_{n})$. For $\,1<i\leq
k$ let $X_{n,i}\subseteq V(G_{n})$ be a maximal $R$-separated subset with
respect to the graph metric $d_{C_{n,i-1}}$. Recall that a set is $R$-separated if
every two distinct points in it has distance at least $R$. By maximality, for
every $x\in V(G_{n})$ there exists $y\in X_{n,i}$ with $d_{C_{n,i-1}}(x,y)\leq
R$. The density of $X_{n,i}$ in a cycle of size at least $R$ is at most $1/R$.
Let  
\[
E_{n}=\left\{  v\in V(G_{n})\mid B_{R+1}(G_{n},v)\ncong B_{R+1}(G,e)\right\}
\text{.}%
\]
be the set of exceptional vertices, where the sofic approximation is bad. 

Let the edge sets 
\[
D_{n,i}=\left\{  (x,s_{i})\mid x\in X_{n,i}\right\}  \text{ (}1\leq i\leq
k\text{)}%
\]
and  
\[
F_{n}=\left\{  (v,s_{i})\mid v\in E_{n},1\leq i\leq k\right\}  \text{.}%
\]
Let the graph 
\[
H_{n}=F_{n}\cup\bigcup\limits_{i=1}^{k}D_{n,i}\text{.}%
\]
That is, we keep all the $s_{1}$-edges in $G_n$ and add the $s_{i}$-edges starting at
$X_{n,i}$ ($1<i\leq k$) and then add all edges starting at the exceptional
vertices. 

We claim that $(H_{n})$ is a rewiring of $(G_{n})$ with bi-Lipschitz distortion at most $(2R+1)^k$ . Since $H_{n}$ is a
subgraph of $G_{n}$, we only have to check one Lipschitz condition. Let $x\in
V(G_{n})$, let $1<m\leq k$ and let $y=xs_{m}$. We claim that $d_{H_{n}%
}(x,y)\leq(2R+1)^{m}$. We prove this by induction on $m$. For $m=1$ this is
trivial. Assume $m>1$. Then there exists $z\in X_{n,m}$ such that $d_{C_{n,m-1}%
}(x,z)\leq R$, that is, $z=xs_{m-1}^{l}$ with some $0\leq\left\vert
l\right\vert \leq R$. If $x\in E_{n}$ is exceptional then $(x,s_{m})\in F_{n}\subseteq H_{n}%
$. Otherwise, the $R+1$ ball around $x$ in $G_{m}$ is isomorphic to the $R+1$
ball in $G$. Since $s_{m-1}$ and $s_{m}$ commute in $\Gamma$, the word
$s_{m-1}^{l}s_{m}s_{m-1}^{-l}$ is a walk of length at most $2R+1$ going from $x$ to $y$ in
$C_{n,m-1}\cup D_{n,m}$. That is, $d_{C_{n,m-1}\cup D_{n,m}}(x,y)\leq2R+1$. By
induction, every $s_{m-1}$-edge in this walk can be substituted by a walk of
length at most $(2R+1)^{m-1}$ in $H_{n}$. This proves the claim and yields
\[
d_{L}(G_{n},H_{n})\leq(2R+1)^{k}%
\]
which implies that $(H_{n})$ is a rewiring of $(G_{n})$. 

Since $s_{i}$ has infinite order, for all $x\notin E_{n}$ the $C_{n,i}$-cycle of
$x$ has size at least $R$. On such cycles the density of $X_{n,i}$ is at most
$1/R$ ($1\leq i\leq k$). This implies
\[
E(H_{n})\leq k\left\vert E_{n}\right\vert +\left\vert V(G_{n})\right\vert
+(k-1)(\left\vert E_{n}\right\vert +\frac{\left\vert V(G_{n})\right\vert
-\left\vert E_{n}\right\vert }{R})
\]
which yields
\[
\frac{\left\vert E(H_{n})\right\vert }{\left\vert V(G_{n})\right\vert }%
\leq1+\frac{k-1}{R}+(2k-1-\frac{k-1}{R})\frac{\left\vert E_{n}\right\vert
}{\left\vert V(G_{n})\right\vert }\text{.}%
\]
Lemma \ref{graf} gives $\left\vert E_{n}\right\vert /\left\vert V(G_{n}%
)\right\vert \rightarrow0$, which gives
\[
e(H_{n})\leq1+\frac{k-1}{R}\text{.}%
\]
By chosing $R$ arbitrarily large, we get $\mathrm{cc}(G_{n})=1$. 
\end{proof} \bigskip

The following corollary of the above proof directly implies Theorem \ref{ccsofapprox}.
\begin{proposition} \label{R}
Let $\Gamma$ be a right angled group generated by a set of elements $S=\{x_1, \ldots, x_k\}$ of infinite order such that
$[x_i,x_{i+1}]=1$ for all $i=1, \ldots, k-1$. Let $(V_n, \sigma_n)$ be a sofic approximation of $\Gamma$ and let $G_{n}=\mathrm{Sch}(V_{n},S,\sigma_{n})$.

For every even integer $R \in \mathbb N$ there exist rewiring graphs $H_n$ of $G_n$ such that $d(G_n,H_n) \leq (2R+1)^k$ and
\[ \frac{\left\vert E(H_{n})\right\vert }{\left\vert V(G_{n})\right\vert }%
\leq1+\frac{k-1}{R}+2k\frac{\left\vert E_{n}\right\vert
}{\left\vert V(G_{n})\right\vert }
\]
where 
\[
E_{n}=\left\{  v\in V(G_{n})\mid B_{R+1}(G_{n},v)\ncong B_{R+1}(\mathrm{Cay}(\Gamma,S),e)\right\}
\text{.}%
\]
In particular
\[ \limsup_{n} \frac{|E(H_n)|}{|V(G_n)|} \leq 1+ \frac{k-1}{R} .\]

\end{proposition}

Theorem \ref{ccsofapprox} follows directly by setting $\epsilon_R= k/R$ and noting that
\[
d(G_n,H_n) \leq (2R+1)^k= O(\epsilon_R^{-k}).
\]

For a subgroup $H$ of $\Gamma$ let $\mathrm{Sch}(\Gamma,H,S)$ denote the
Schreier diagram of $\Gamma$ with respect to $H$ and $S$. Recall that the
vertex set of $\mathrm{Sch}(\Gamma,H,S)$ is the right coset space $\Gamma/H$
and for every $s\in S$ and $x\in\Gamma/H$ there is an $s$-labeled edge going
from $x$ to $xs$. The metric coming from a Schreier diagram is defined by
forgetting its labels and treating it as a simple graph.

We will use the language of measured groupoids to analyze the connection
between the rank gradient and the combinatorial cost. This goes back to
\cite{miknik} where the first and last authors expressed the rank gradient of
a chain from the cost of the corresponding profinite measured groupoid.

Recall that a measured groupoid $\mathcal{G}$ is a Borel probability space
$\mathcal{G}^{0}$ together with a\ Borel space of arrows $\mathcal{G}^{1}$,
maps $i,t:\mathcal{G}^{1}\rightarrow\mathcal{G}^{0}$ (the initial and terminal
map) and a composition operation such that $\mathcal{G}$ is a groupoid with
respect to these operations. We also assume that the underlying measure $\mu$
is invariant under the arrows.

Let the countable group $\Gamma$ act on the Borel probability space $(X,\mu)$
by measure preserving maps. Set $\mathcal{G}^{0}=(X,\mu)$ and $\mathcal{G}%
^{1}=\left\{  (x,g)\mid x\in X,g\in\Gamma\right\}  $. Let $i((x,g))=x$ and
$t((x,g))=xg$. Let the product $(x,g)(y,h)$ be defined if $y=xg$, in this case
it equals $(x,gh)$ and the inverse of $(x,g)$ is defined to be $(xg,g^{-1})$.
This defines the measured groupoid of the action. Using the initial map the
measure $\mu$ can be naturally lifted to a Borel measure $\widetilde{\mu}$ on
$\mathcal{G}^{1}$.

Let the \emph{cost} of a measured groupoid $\mathcal{G}$ be
\[
\mathrm{c}(\mathcal{G})=\inf_{\substack{Y\subseteq\mathcal{G}^{1}\text{
Borel}\\Y\text{ generates }\mathcal{G}}}\widetilde{\mu}(Y)
\]
\newline the infimum of measures of generating Borel subsets. Here generation is meant in the usual algebraic sense. 
The notion of cost for measurable equivalence relations has been introduced by Levitt
\cite{levitt} and analyzed in depth by Gaboriau \cite{gabor}.

A case of particular interest in this paper is when $\Gamma$ acts on the right
coset space $\Gamma/H$ for some finite index subgroup $H$ of $\Gamma$. Here
$\mu$ and $\widetilde{\mu}$ are just counting measures, normalized by the
index of $H$ in $\Gamma$. In \cite{miknik} the following is proved. Since we
use a somewhat different language in \cite{miknik}, we sketch a proof here.

\begin{lemma}
\label{measuredgr}Let $\Gamma$ be a countable group and let $H$ be a subgroup
of finite index in $\Gamma$. Let $\mathcal{G}$ be the measured groupoid of the
right coset action of $\Gamma$ on $\Gamma/H$. Then we have
\[
r(\Gamma,H)=\frac{d(H)-1}{|G:H|}=\mathrm{c}(\mathcal{G})-1\text{.}%
\]

\end{lemma}

\noindent\textbf{Proof. }If $R$ generates $H$ then it is easy to see that
\[
\left\{  (H,r)\mid r\in R\right\}
\]
together with any subset of $\mathcal{G}^{1}$ that connects $\mathcal{G}^{0}$
(e.g. a spanning tree) will generate $\mathcal{G}$. This gives
\[
\mathrm{c}(\mathcal{G})\leq\frac{1}{\left\vert \Gamma:H\right\vert }\left(
\left\vert R\right\vert +(\left\vert \Gamma:H\right\vert -1)\right)
\]
which gives
\[
r(\Gamma,H)\geq\mathrm{c}(\mathcal{G})-1\text{.}%
\]

Let $Y\subseteq\mathcal{G}^{1}$ be a subset. Look at the graph $G$ on
$\Gamma/H$ defined by $Y$. Note that $G$ may have multiple edges and loops.
The fundamental group $\pi_{1}$ of $G$ rooted at the vertex $H$ admits a
natural homomorphism $\Phi$ to $H$ by evaluating loops by taking the product
of edge labels coming from $\mathcal{G}^{1}$. It is easy to see that $Y$
generates $\mathcal{G}$ if and only if $G$ is connected and $\Phi$ is
surjective. So, if $Y$ generates $\mathcal{G}^{1}$, we have
\[
d(\pi_{1})\geq d(H)\text{.}%
\]

Let $Y$ be a generating subset of minimal measure $\mathrm{c}(\mathcal{G})$.
The fundamental group $\pi_{1}$ is a free group of rank $\left\vert
Y\right\vert -\left\vert \Gamma:H\right\vert +1$, which yields
\[
r(\Gamma,H)=\frac{d(H)-1}{\left\vert \Gamma:H\right\vert }\leq\frac{\left\vert
Y\right\vert }{\left\vert \Gamma:H\right\vert }-1=\mathrm{c}(\mathcal{G}%
)-1\text{.}%
\]

The Lemma holds. $\square$\bigskip

Let $\mathcal{G}$ be a measured groupoid of a p.m.p. action of $\Gamma$ on
$(X,\mu)$. Let the trivial subset
\[
N=\left\{  (x,1)\in\mathcal{G}^{1}\mid x\in X\right\}  \text{.}%
\]

\begin{definition}
\label{groupoiddist}For two Borel subsets $A,B\subseteq\mathcal{G}^{1}$ let
the distance between them be defined as
\[
d_{L}(A,B)=\inf\left\{  k\in\mathbb{N}\mid A\subseteq(B\cup B^{-1}\cup
N)^{k}\text{ and }B\subseteq(A\cup A^{-1}\cup N)^{k}\right\}  \text{.}%
\]

\end{definition}

Of course, $d_{L}$ may be infinite, even when both $A$ and $B$ generate
$\mathcal{G}$. \medskip

Let $H$ be a subgroup of $\Gamma$ of finite index. Let $\mathcal{G}$ denote
the measured groupoid defined by the right action of $\Gamma$ on $\Gamma/H$,
with underlying measure $\mu$. Let $G=(V,F)$ be an abstract directed graph defined on
the vertex set $V=\Gamma/H$. By a $\Gamma$-labeling of $G$ we mean a map
$\phi:F\rightarrow\Gamma$ such that for all $f=(x,y)\in F$, we have
$x\phi(f)=y$. In other terms, $\phi$ is a map from $F$ to $\mathcal{G}^{1}$
which is trivial on the $\mathcal{G}^{0}$ level. When $G$ is a Schreier graph,
the natural $\Gamma$-labeling means the labeling coming from the coset action.

The following lemma controls how a rewiring of a Schreier graph (viewed as an
abstract graph) can be labeled and extended to be a generator for the measured
groupoid. This lemma is essential to proving Theorem \ref{rgcc} and the
estimates on the torsion homology.

\begin{lemma}
\label{lenyeg}Let $\Gamma$ be a group generated by the finite symmetric set
$S$. Let $H$ be a subgroup of $\Gamma$ of finite index with measured groupoid
$\mathcal{G}$ and let $K$ be the uniform random conjugate of $H$ in $\Gamma$.
Let $G=(V,E)=\mathrm{Sch}(\Gamma,H,S)$ with the natural $\Gamma$-labeling
$\varphi$ and let $G^{\prime}=(V,F)$ be a graph with $D=d_{L}(G,G^{\prime})$.
Then there exists a $\Gamma$-labeling $\phi$ of $G^{\prime}$ and a subset
$I\subseteq\mathcal{G}^{1}$ of measure
\[
\tilde \mu(I)\leq\sum\limits_{g\in S^{D^{2}+1}\backslash\{1\}}\mathcal{P}(g\text{ is
in }K)
\]
such that
\[
d_{L}(\varphi(E),\phi(F)\cup I)\leq D^{2}+1\text{.}%
\]
In particular, we have
\[
r(\Gamma,H)\leq\frac{\left\vert F\right\vert }{\left\vert V\right\vert
}-1+\sum\limits_{g\in S^{D^{2}+1}\backslash\{1\}}\mathcal{P}(g\text{ is in
}K)\text{.}%
\]

\end{lemma}

\begin{proof} By the definition of $D$, for each $f=(x,y)\in F$,
there exists $k\leq D$ and a walk $e_{1},\ldots,e_{k}\in E$ from $x$ to $y$.
Choose such a walk and let
\[
\phi(f)=\varphi(e_{1})\cdots\varphi(e_{k})\text{.}%
\]

Let $e=(x,y)\in E$. Then there exists $k\leq D$ and a walk $f_{1},\ldots
,f_{k}\in F$ from $x$ to $y$. Choose such a walk. Let
\[
g(e)=\varphi(e)\phi(f_{k})^{-1}\cdots\phi(f_{1})^{-1}\in\Gamma\text{.}%
\]
Let $I\subseteq\mathcal{G}$ defined by
\[
I=\left\{  (x,g(e))\mid e=(x,y)\in E,g(e)\neq1\right\}  \text{ }%
\]
Note that $g(e)$ stabilizes $x$. Let
\[
J=\phi(F)=\left\{  (x,\phi(f))\mid f=(x,y)\in F\right\}  \text{.}%
\]

We claim that $I\cup J\subseteq\varphi(E)^{D^{2}+1}$. For $f\in F$ we have
$\phi(f)=\varphi(e_{1})\cdots\varphi(e_{k})$ which implies $J\subseteq
\varphi(E)^{D}$. For $e=(x,y)\in E$ with $g(e)\neq1$ we have $g(e)=\varphi
(e)\phi(f_{k})^{-1}\cdots\phi(f_{1})^{-1}$. Since $\phi(f_{k})\in E^{D}$ we
get $(x,g(e))\in\varphi(E)^{D^{2}+1}$ and so $I\subseteq\varphi(E)^{D^{2}+1}$.
The claim holds.

We now claim that $\varphi(E)\subseteq(I\cup J)^{D+1}$. Let $e=(x,y)\in E$. If
$g(e)=1$ then $\varphi(e)=\phi(f_{1})\cdots\phi(f_{k})$ and thus
$\varphi(e)\in J^{D}$. If $g(e)\neq1$ then $\varphi(e)=g(e)\phi(f_{1}%
)\cdots\phi(f_{k})$ and thus $\varphi(e)\in J^{D+1}$. The claim holds.

The two claims together give
\[
d_{L}(\varphi(E),I\cup J)\leq D^{2}+1\text{.}%
\]
Since $\varphi(E)$ generates $\mathcal{G}$, this also implies that $I\cup J$
generates $\mathcal{G}$.

For every $e$, $g(e)$ stabilizes $x$ and can be written as a word of length at
most $D^{2}+1$ in $S$. Hence, the measure
\[
\widetilde{\mu}(I)\leq\sum\limits_{g\in S^{D^{2}+1}\backslash\{1\}}%
\mathcal{P}(g\text{ is in }K)
\]
and we have
\[
\widetilde{\mu}(J)=\left\vert F\right\vert /\left\vert V\right\vert \text{.}%
\]

By Lemma \ref{measuredgr}, $r(\Gamma,H)+1$ equals the minimal $\widetilde{\mu
}$-measure of a generating subset, thus we have
\[
r(\Gamma,H)+1\leq\widetilde{\mu}(I\cup J)\leq\frac{\left\vert F\right\vert
}{\left\vert V\right\vert }+\sum\limits_{g\in S^{d^{2}+1}\backslash
\{id\}}\mathcal{P}(g\text{ is in }K)
\]
The lemma holds. \end{proof} \bigskip

We are ready to prove Theorem \ref{rgcc}. \medskip  

\noindent\textbf{Proof of Theorem \ref{rgcc}.} Let $G_{n}=\mathrm{Sch}(\Gamma,\Gamma_{n},S)$ and let
$cc=\mathrm{cc}(G_{n})$. Let $K_{n}$ be the IRS with distribution $\mu
_{\Gamma_{n}}$.

Let $\varepsilon>0$. Then there exists a rewiring $(G_{n}^{\prime})$ of
$(G_{n})$ such that $e(G_{n}^{\prime})\leq cc+\varepsilon$. Let $d>0$ such
that $d_{L}(G_{n},G_{n}^{\prime})<d$ ($n\geq1$). Let $W=S^{d^{2}+1}%
\backslash\{id\}$. Lemma \ref{lenyeg} implies that for all $n\geq1$ we have%

\[
r(\Gamma,\Gamma_{n})\leq\frac{\left\vert E(G_{n}^{\prime})\right\vert
}{\left\vert G_{n}\right\vert }-1+\sum\limits_{g\in W}\mathcal{P}(g\text{ is
in }K_{n})\text{.}%
\]
Since $(\Gamma_{n})$ is Farber, for each $g\in W$, we have
\[
\lim_{n\rightarrow\infty}\mathcal{P}(g\text{ is in }K_{n})=0\text{.}%
\]
So, taking a liminf on both sides and using that $\mathrm{r}(\Gamma,(\Gamma_{n}))$ converges gives us
\[
\mathrm{RG}(\Gamma,(\Gamma_{n}))\leq\mathrm{cc}(\mathrm{Sch}(\Gamma,\Gamma
_{n},S))-1
\]
and the inequality is proved.

When $(\Gamma_{n})$ forms a chain, then the reverse inequality is proved in
\cite[Theorem 1.2]{elek}. We give a short proof here using the measured groupoid language.

Let $\mathcal{G}_{n}$ be the measured groupoid associated to the action of
$\Gamma$ on $\Gamma/\Gamma_{n}$. Assume that $r(\Gamma,\Gamma_{n_{0}})\leq R$
for some $n_{0}$. Let $F\subseteq$ $\mathcal{G}_{n_{0}}$ be a generating
subset of measure at most $R+1$ and let $E\subseteq$ $\mathcal{G}_{n_{0}}$ be
the standard generating set coming from $S$. Let $d>0$ such that
$F^{d}\supseteq E$ and $E^{d}\supseteq F$, where $Y^{d}$ denotes all words of
length at most $d$ in $\mathcal{G}_{n_{0}}$.

Let $\Phi_{n}:\mathcal{G}_{n}\rightarrow\mathcal{G}_{n_{0}}$ $(n\geq n_{0})$
denote the standard factor map from $\mathcal{G}_{n}$ to $\mathcal{G}_{n_{0}}%
$. This associates a $\Gamma_{n}$-coset to the unique $\Gamma_{n_{0}}$-coset
containing it and is a groupoid homomorphism. Let $F_{n}=\Phi_{n}^{-1}(F)$ and
$E_{n}=\Phi_{n}^{-1}(E)$. Then $F_{n}$ has the same measure as $F$ and we have
$F_{n}^{d}\supseteq E_{n}$ and $E_{n}^{d}\supseteq F_{n}$. Let $G_{n}^{\prime
}$ be the graph defined on $\Gamma/\Gamma_{n}$ by the arrows $F_{n}$. Then we
have
\[
d_{L}(G_{n}^{\prime},\mathrm{Sch}(\Gamma,\Gamma_{n},S))\leq d
\]
and so $G_{n}^{\prime}$ is a rewiring of $\mathrm{Sch}(\Gamma,\Gamma_{n},S)$.
This implies $\mathrm{cc}(\mathrm{Sch}(\Gamma,\Gamma_{n},S))\leq R$ and proves
the reverse inequality.

The Theorem holds. $\square$\bigskip

\noindent\textbf{Remark. }We do not have an example for a strict inequality in
Theorem \ref{rgcc}. Maybe equality holds in general, but it is not clear how
to get a cheap rewiring from a cheap generating set for the measured groupoid.
A priori, such cheap generating sets (or any rewirings of them) may lie quite deep
in the group structure, in which case generation may happen very slowly.

\begin{theorem} \label{ra}
Let $\Gamma$ be a right angled group and let $(\Gamma_{n})$ be a Farber
sequence of subgroups of finite index in $\Gamma$. Then
\[
\mathrm{RG}(\Gamma,(\Gamma_{n}))=0\text{.}%
\]

\end{theorem}

\noindent\textbf{Proof. } By passing to a subsequence we can assume that $\mathrm{r}(\Gamma,(\Gamma_{n}))$ converges. 
Let $S$ be a finite generating set of $\Gamma$ and
let $\Phi:F_{S}\rightarrow\Gamma$ be the canonical homomorphism with kernel
$N=\ker\Phi$. Let $H_{n}=\Phi^{-1}(\Gamma_{n})$ $(n\geq1)$. Since $H_{n}$ is
Farber, by Lemma \ref{ujirs} the coset actions $\sigma_{n}$ of $\Gamma$ on
$\Gamma/\Gamma_{n}$ give a sofic approximation for $\Gamma$. By Theorem
\ref{rightangcc} we have $\mathrm{cc}(\mathrm{Sch}(\Gamma,\Gamma_{n},S))=1$.
Theorem \ref{rgcc} yields
\[
\mathrm{RG}(\Gamma,(\Gamma_{n}))\leq\mathrm{cc}(\mathrm{Sch}(\Gamma,\Gamma
_{n},S))-1=0\text{.}%
\]
The Theorem holds. $\square$ \medskip

\noindent\textbf{Proof of Theorem \ref{main1}}. 
By Theorem \ref{mu_n}, the sequence $(\Gamma_n)$ is Farber. Applying Theorem \ref{ra} gives $\mathrm{RG}(\Gamma,(\Gamma_{n}))=0$. $\square$ \medskip


\section{Growth of homology torsion of right angled groups} \label{torsion}

In this section we give upper estimates on the torsion homology and prove
Theorem \ref{subexptorsion}\textbf{. }

For a finite CW-complex $M$ let $\pi_{1}M$ denote its fundamental groupoid. This 
is a measured groupoid defined as follows. Let us put the uniform probability measure on 
the vertices of $M$. Let $P(M)$ denote the set of paths in $M$. This is a groupoid under 
concatenation. Then $\pi_1M$ is the quotient of $P(M)$ modulo homotopy.  
Since the set of directed edges $E(M)$ in $M$ generate $P(M)$, their homotopy classes 
generate $\pi_{1}M$.

We now describe how to rewire a complex to another generating set of its
fundamental groupoid. 

Let $M$ be a finite CW-complex, let $\mathcal{G}=\pi_{1}M$ and let $E=E(M)$.
Let $F\subseteq\mathcal{G}$ be another generating set. We define the
\emph{rewired complex }$N$ \emph{with respect to }$F$ as follows. Let
$N_{0}=M_{0}$ and $N_{1}=F$. For every $e\in E$ choose a path $\varphi
(e)=(f_{1},\ldots,f_{k})$ in $N$ such that $e=f_{1}\cdots f_{k}$. Also, for
every $f\in F$ choose a path $\phi(f)=(e_{1},\ldots,e_{s})$ in $M$ such that
$f=e_{1}\cdots e_{s}$ (where $k$ and $s$ may depend on $e$ and $f$ respectively). 
The maps $\varphi$ and $\phi$ extend to $P(M)$ and $P(N)$ by
concatenation. For every disc bounded by the path $p$ in $M$ add the
disc bounded by $\varphi(p)$ to $N$ (type I). For every $f\in F$ add the
disc bounded by $\varphi\phi(f)f^{-1}$ to $N$ (type II).

Note that the rewired complex $N$ does depend on the choices $\varphi$ and $\phi$.
\begin{proposition}
\label{haha}Using the above notations, we have $\pi_{1}M \cong \pi_{1}N$.
\end{proposition}

\begin{proof} The maps $\varphi:P(M)\rightarrow P(N)$ and
$\phi:P(N)\rightarrow P(M)$ are well defined groupoid homomorphisms that preserve the
starting and endpoints of paths.

For a path $p\in P(N)$ or $p\in P(M)$ let $\Delta(p)$ be the product of its
edges (as elements of $\mathcal{G}^{1}$). Then $\Delta$ is a groupoid
homomorphism from each of $P(N)$ and $P(M)$ to $\mathcal{G}^1$, and by the definition of $\varphi$ and $\phi$, we have
\[
\Delta(\varphi(p))=\Delta(p)\text{ (}p\in P(M)\text{) and }\Delta
(\phi(p))=\Delta(p)\text{ (}p\in P(N)\text{).}%
\]
Since $\mathcal{G}=\pi_{1}M$, two paths $p_{1},p_{2}\in P(M)$ are homotopic in
$M$ if and only if $\Delta(p_{1})=\Delta(p_{2})$.

We claim that $\varphi$ and $\phi$ respect the homotopy defined by the 2-cells of $M$ and $N$ and so
induce well-defined homomorphisms $\varphi:\pi_{1}N\rightarrow\pi_{1}M$ and $\phi:\pi
_{1}M\rightarrow\pi_{1}N$. Indeed, if $p\in P(M)$
bounds a disc, then $\varphi(p)$ is a type I disc in $N$. So $\varphi$
respects homotopy. If the loop $p\in P(N)$ starting at $x$ bounds a type I
disc in $N$, then $p=\varphi(p^{\prime})$ for some disc $p^{\prime}\in
M$, which implies
\[
\Delta(\phi(p))=\Delta(p)=\Delta(\varphi(p^{\prime}))=\Delta(p^{\prime})=1_{x}%
\]
that is, $\phi(p)$ is nullhomotopic. If the loop $p\in P(N)$ starting at $x$
bounds a type II disc in $N$, then $p=\varphi\phi(f)f^{-1}$ for some edge
$f$ in $M$. By the above, we have
\[
\Delta(\phi(p))=\Delta(\varphi\phi(f)f^{-1})=\Delta(\varphi\phi(f))\Delta
(f^{-1})=1_{x}%
\]
that is, $\phi(p)$ is nullhomotopic. So $\phi$ respects homotopy. The claim
holds and implies that $\varphi:\pi_{1}N\rightarrow\pi_{1}M$ and $\phi:\pi
_{1}M\rightarrow\pi_{1}N$ are well defined homomorphisms.

Let $p\in P(N)$. Then we have $\Delta(\phi\varphi(p))=\Delta(p)$, thus
$\phi\varphi(p)$ and $p$ are homotopic. Let $p\in P(M)$. Then by the type II
relations, $\varphi\phi(p)$ is homotopic to $p$. We get that $\phi\varphi$ and
$\varphi\phi$ are the identity maps on $\pi_{1}M$ and $\pi_{1}N$ and so
$\varphi:\pi_{1}N\rightarrow\pi_{1}M$ is an isomorphism. \end{proof} \medskip

In a more algebraic way, we assert above that the homotopy classes and the
$\Delta$-classes on $P(N)$ are equal, which establishes the isomorphism
between $\pi_{1}N$ and $\pi_{1}M$.

The following is a stardard lemma in linear algebra.

\begin{lemma}
\label{torlemma}\label{tor}Let $m,d,l \in\mathbb{N}$ and let $v_{1}, \ldots,
v_{l} \in\mathbb{Z}^{m}$. Suppose that each of the vectors $v_{i}$ is such
that the sum of the absolute values of its coordinates is at most $b$. Let
$A=\mathbb{Z}^{m}/ \sum_{i=1}^{l} \mathbb{Z }v_{i}$. Let $\mathrm{{trs}(A)}$
denote the size of the torsion subgroup of $A$. Then $\mathrm{trs}(A) \leq b^{m}$.
\end{lemma}

\noindent\textbf{Proof.} Let $X$ be the matrix whose rows are the vectors
$v_{1},\ldots,v_{l}$. By \cite{Jacob} Theorem 3.9, $\mathrm{{trs}(A)}$ is the
greatest common divisor of all non-zero minors of $A$. So to prove the lemma
we need to bind any such minor. Therefore we may assume that $X$ is
non-singular and $m=l$. From geometric consideration in the Euclidean space
$\mathbb{R}^{m}$ we know that $|\det X|$ is at most the product of the lengths
of the row vectors of $X$. So $|\det X|\leq b^{m}$ and the lemma follows.
$\square$ \medskip

In the following Lemma, we show that for a finitely presented group, the
torsion homology growth stays bounded over finite index subgroups. Note that
by \cite{kkn} this is false already in the realm of finitely generated solvable groups.

\begin{lemma}
\label{simplebound} Let $\Gamma= \langle X; R \rangle$ be a finitely presented
group. Let $|X|=d$ and let $b$ be the maximal length of the relations in $R$.
Let $H$ be a subgroup of index $n$ in $\Gamma$. Then
\[
\frac{\ln\mathrm{trs} H}{n} \leq d \ln b .
\]

\end{lemma}

\noindent\textbf{Proof.} Consider the Schreier presentation $\langle
X^{\prime};R^{\prime}\rangle$ for $H$ obtained from the presentation of
$\Gamma$. We have $|X^{\prime}|=n(d-1)+1<nd$ and in addition all relations in
$R^{\prime}$ have length at most $b$ considered as words in $X^{\prime}$. By
considering the Abelianization of the presentation $\langle X^{\prime
};R^{\prime}\rangle$ and using Lemma \ref{torlemma}, we get
\[
\mathrm{trs}(H)\leq b^{|X^{\prime}|}<b^{dn}%
\]
and the result follows. $\square$ \medskip

We are ready to prove the main theorem of this section. We will make use of the following definition:

\begin{definition}
 A graph sequence $(G_n)$ has {\it subexponential distortion} if there are $d_{r}$ and
$\varepsilon_{r}$ such that $\epsilon_r\to 0$ and
\[
\lim_{r\rightarrow\infty}\varepsilon_{r}\ln d_{r}=0
\]
and for every $r$ there is a rewiring $G_{n}^{(r)}$ of $G_{n}$ such that
$d_{L}(G_{n},G_{n}^{(r)})\leq d_{r}$ and
\[
\lim\sup_{n\rightarrow\infty}\frac{|E(G_{n}^{(r)})|}{|V(G_n^r)|}<1+\varepsilon_{r}\text{.}%
\]

\begin{remark}\label{rem:limsup}
Note that here we used $\lim\sup$ instead of $\lim\inf$ as in the original definition of combinatorial cost for obvious purposes.
\end{remark}

\end{definition}

\medskip

\noindent\textbf{Proof of Theorem \ref{subexptorsion}.} Let $(S,R)$ be a
finite presentation of $\Gamma$. Let $M$ be the corresponding presentation
$2$-complex, with one vertex, $\left\vert S\right\vert $ edges labeled by $S$
and the relations glued on as disks. Let $b$ be the maximal size of a disc
in $M$. Let
\[
M_{n}=\widetilde{M}/\Gamma_{n}%
\]
be the lift of $M$ with respect to $\Gamma_{n}$. The $1$-skeleton of $M_{n}$
equals the Schreier graph $G_{n}=\mathrm{Sch}(\Gamma,\Gamma_{n},S)$. Since
$\Gamma_{n}$ is a Farber sequence, $G_{n}$ is a sofic approximation for
$\Gamma$. Let $\mathcal{G}_{n}$ be the measured groupoid of the coset action
of $\Gamma$ on $\Gamma/\Gamma_{n}$ with underlying measure $\mu_{n}$ and let
$E_{n}$ be the standard generating set of $\mathcal{G}_{n}$.

By the subexponential distortion assumption, there are sequences $d_{r}$ and
$\varepsilon_{r}$ such that
$\lim_{r\rightarrow\infty}\varepsilon_{r}\ln d_{r}=0$
and for every $r$ there is a rewiring $G_{n}^{(r)}$ of $G_{n}$ such that
$d_{L}(G_{n},G_{n}^{(r)})\leq d_{r}$ and
$$
\frac{|E(G_{n}^{(r)})|}{|V(G_n^r)|}<1+\varepsilon_{r}+\delta_n^{(r)},
$$
for some $\delta_n^{(r)}>0$ which tends to $0$ with $n$ for any given $r$.

Let $K_{n}$ be
the uniform random conjugate of $\Gamma_{n}$ in $\Gamma$ and let
\[
\gamma_{n}^{(r)}=\sum\limits_{g\in S^{d_{r}^{2}+1}\backslash\{1\}}\mathcal{P}%
(g\text{ is in }K_{n})\text{.}%
\]
For a fixed $r$ let $H_{n}=G_{n}^{(r)}$. By Lemma \ref{lenyeg} there exist
generating sets $F_{n}$ of $\mathcal{G}_{n}$ of measure
\[
\widetilde{\mu}_{n}(F_{n})\leq1+\varepsilon_{r}+\delta_{n}^{(r)}+\gamma_{n}^{(r)}%
\]
such that %

\[
d_{L}(E_{n},H_{n})\leq d_{r}^{2}+1\text{.}%
\]
Let $N_{n}$ be the rewiring of the complex $M_{n}$ with respect to $F_{n}$,
such that the rewiring functions $\varphi$ and $\phi$ always 
take the edges of the respective complex to paths of length at most $d_{r}^{2}+1$. Taking into account both type I and II $2$-discs
in $N_{n}$, the maximal size of a $2$-disc in $N_{n}$ is at most
\[
\max\left\{  b(d_{r}^{2}+1),(d_{r}^{2}+1)^{2}+1\right\}  \leq4bd_{r}%
^{4}\text{.}%
\]

By Proposition \ref{haha} the measured groupoid $\pi_{1}N_{n}\cong
\pi_{1}M_{n}=\mathcal{G}_{n}$. This means that for any vertex $x$ of $N_{n}$,
the fundamental group $\pi_{1}(N_{n},x)\cong \Gamma_{n}$. That is, $N_{n}$
is a presentation complex for $\Gamma_{n}$. Choosing a spanning tree in the
$1$-skeleton of $N_{n}$ and contracting it to a point gives that $\Gamma_{n}$ admits a
presentation with number of generators at most
\[
\left(  \widetilde{\mu}_{n}(F_{n})-1\right)  \left\vert \Gamma:\Gamma
_{n}\right\vert \leq\left(  \varepsilon_{r}+\delta_{n}^{(r)}+\gamma_{n}^{(r)}\right)
\left\vert \Gamma:\Gamma_{n}\right\vert
\]
where all the relations have size at most $4bd_{r}^{4}$. This yields
\[
\frac{\log\mathrm{trs}\Gamma_{n}}{[\Gamma:\Gamma_{n}]}\leq\left(
\varepsilon_{r}+\delta_{n}^{(r)}+\gamma_{n}^{(r)}\right)  \log(4bd_{r}^{4})\text{.}%
\]
Since $(\Gamma_{n})$ is a Farber sequence, we have $\lim_{n\rightarrow\infty
}\gamma_{n}^{(r)}=0$ and we get
\[
\limsup_{n\rightarrow\infty}\frac{\log\mathrm{trs}\Gamma_{n}}{[\Gamma:\Gamma
_{n}]}\leq\varepsilon_{r}\log(4bd_{r}^{4})\text{.}%
\]
Letting $r$ tend to infinity yields
\[
\lim_{n\rightarrow\infty}\frac{\log\mathrm{trs}\Gamma_{n}}{[\Gamma:\Gamma
_{n}]}=0\text{.}%
\]
The Theorem follows. $\square$ \medskip

\subsection{Explicit estimates with and without the CSP} \label{estimates}

In this section we will focus on the case when $\Gamma$ is an arithmetic group. We show that when $\Gamma$ has the generalized congruence subgroup property (CSP) there is a polynomial bound on the size of the torsion in homology.

\begin{proposition}
Let $G$ be a semi-simple algebraic group defined over a number field $K$ and let $\Gamma$ be an arithmetic subgroup of $G$. Suppose that $\Gamma$ has CSP. Then there is a number $c>0$ (depending on $\Gamma$) such that
if $H \leq \Gamma$ has index $n$ then $|H^{ab}| < n^c$.
\end{proposition}

\begin{proof} By \cite{LM} the arithmetic group $\Gamma$ has CSP if and only if $\Gamma$ has polynomial representation growth, that is there exists a number $c_0>0$ such that $\Gamma$ has at most $n^{c_0}$ irreducible representations of degree at most $n$. Suppose now that $H$ is a subgroup of index $n$. We claim that $|H^{ab}| \leq n^{c_0+1}$. This is provided by Lemma 2.2 of \cite{LM} but for completeness we give a proof. Suppose that $|H^{ab}|>n^{c_0+1}$, this means that $H$ has more that $n^{c_0+1}$ one dimensional characters $\rho_i$. Let $\phi_i$ be the induced character $Ind_H^G(\rho_i)$, this is a character of $G$ of degree $n$. We claim that each $\phi_i$ can share an
irreducible constituent with at most $n-1$ other characters $\phi_j$. Indeed if $\phi_i$ and $\phi_j$ share an irreducible character of $G$ then by Frobenius reciprocity $\rho_j$ is a summand of the restricted character $(\phi_i)|_N$. However the degree of $(\phi_i)|_N$ is $n$ and so it cannot have more that $n$ irreducible constitutents. The claim follows. This means that the total number of different irreducible characters of $G$ appearing as constituents of all $\phi_i$ is more than $n^{c_0+1}/n=n^{c_0}$, contradiction.
\end{proof}

When CSP fails or is not known to hold the proof of Theorem \ref{tortheorem} gives a much weaker, but still explicit bound.

\begin{theorem} Let $\mathcal G$ be a simple algebraic group defined over a number field $K$ and let $\Gamma$ be a cocompact right angled arithmetic subgroup of $\mathcal G$. There are constants $\alpha>0$ and $A>0$ such that for any congruence subgroup $H$ of index $n$ in $\Gamma$ 
\[ \log \mathrm{trs} H  \leq \frac{A\cdot n}{(\log n)^{\alpha}}. \]

\end{theorem}

\begin{proof}

Fix a right angled generating set $S$ of $\Gamma$ of size $k$. For a graph $L$, $r>0$  and vertex $v$ of $L$, by $B_L(R,v)$ we denote the ball of radius $r$ centered at $v$ in the edge metric of $L$.

We need to estimate the contribution to the edge cost from the bad vertices in the sofic approximation in the proof of Theorem \ref{tortheorem}.

By the discrete version of Theorem 5.2 of \cite{samurai}  there are constants $c$ and $\mu>0$ (depending on $\Gamma$) such that if $r>0$ and $H$ is a congruence subgroup of $\Gamma$ with $n=[\Gamma:H]$ then 

\begin{equation} \label{est}  |\{  v \in \Gamma/H \ | \ B_{\rm{Sch}(\Gamma,H,S)}(r,v)  \not \simeq B_{Cay(\Gamma,S)}(r,v) \}| \leq e^{Cr} n^{1-\mu} .\end{equation}

Recall that $k$ is the size of the right angled generating set $S$ of $\Gamma$. Let $R$ be the even integer closest to $\beta (\log n)^{1/2k}$ for some small constant $\beta>0$ such that $\max\{(3\beta)^{2k}, \beta\}< \mu/2C$. Proposition \ref{R} provides a rewiring of $\rm{Sch}(\Gamma,H,S)$ of distance at most $d=(2R+1)^k$ and edge cost at most $1+(k-1)/R + \delta_n$ where $\delta_n \leq 2k |E_n|/n$ and $E_n$ is the set of vertices of $\rm{Sch}(\Gamma,H,S)$ where the injectivity radius is less than $R+1$. From (\ref{est}) we obtain $|E_n| \leq e^{C(R+1)} n^{1-\mu}=O(n^{1-\mu/2})$ since $R \leq \beta \log n$ and $e^{CR} \leq n^{C\beta} \leq n^{\mu/2}$. Similarly (\ref{est}) shows that the number of  bad vertices of $\rm{Sch}(\Gamma,H,S)$ with injectivity radius less than $d^2+1<(3R)^{2k}$ is at most $O(n^{1-\mu/2})$. So the contribution  from the bad vertices
\[ \gamma_n=\sum\limits_{g\in S^{d^{2}+1}\backslash\{1\}}\mathcal{P}(g\text{ is in
}K) \] in the proof of Theorem \ref{subexptorsion} is at most $O(n^{-\mu/2})$.

The proof of Theorem \ref{tortheorem} now gives that $H$ can be presented by $x$ elements, where
\[ x=n((k-1)/R + \delta_n + \gamma_n)=O(n/(\log n)^{1/2k})\] and relations of length at most $(3R)^{4k}= O((\log n)^{2})$. Therefore by Lemma \ref{torlemma} we obtain
\[ \log |\mathrm{trs}(H^{ab})| \leq x \log ((3R)^{4k}))= O\left(\frac{n}{(\log n)^{1/2k}}\right)O(\log \log n)=O\left(\frac{n \log \log n}{(\log n)^{1/2k}}\right). \] The Theorem follows with any $\alpha<1/2k$. \end{proof}

We note that in case $\Gamma$ has property FAb, (e.g. when $\Gamma$ has property T) we can simplify the above argument because $H$ contains a principal congruence subgroup $N$ of index at most $n^b$ for some constant $b$. Then it is sufficient to bind $|N^{ab}|$ and the above argument is simpler: For normal subgroups the avoidance of bad vertices in the sofic approximation is equivalent to the requirement that $N$ intersects the ball of radius $(2R+1)^{2k}$ in $\mathrm{Cay}(\Gamma,S)$ trivially. This holds when $R = O((\log n)^{1/2k})$ and we reach essentially the same bound as above.


\section{Right angled lattices in higher rank Lie groups} \label{sec:coco}

In this section we construct the lattices $\Gamma$ from Theorems  \ref{coco} and \ref{SO(7,2)}. The main idea is to make sure that $\Gamma$ has a semisimple element $g \not =1$ which is not regular (i.e. has repeated eigenvalues). The (arithmetic) Borel density theorem then can be applied to ensure that $C_\Gamma(g)$ contains elements of infinite order and in this way obtain a right angled generating set for some finite index subgroup of $\Gamma$.

\subsection{Co-compact Right angled lattices in orthogonal groups}
Let $d\in\mathbb{N}$ be a square-free natural number, set $K=\mathbb{Q}(\sqrt{d})$ and denote its ring of integers
$\mathcal{O}_{K}$. Let $n \geq 7$ and define
\[
f(x_{1},\ldots,x_{n+2})=x_{1}^{2}+\ldots+x_{n}^{2}-\sqrt{d}(x_{n+1}^{2}+x_{n+2}^{2})%
\]
be a quadratic form on $V=K^{n+2}$. Let $\mathbb{G}=SO(f)$ and $\Gamma
=\mathbb{G}(\mathcal{O}_{K})$. We will construct a finite set $\{b_1, \ldots, b_m\}$ such that $\langle b_i,b_{i+1}\rangle \simeq \mathbb Z^2$ for each $i=1, \ldots, m-1$, and such that $\langle b_1,\ldots,b_m\rangle$ is of finite index in $\Gamma$.

 \bigskip

Denote by $V$ the quadratic module $(K^{n+2},f)$. As $K$ is a subfield of $\mathbb{R}$ the sign $\pm$ is defined for elements in $K$.
For subspaces $A,B$ of $V$ we will write $A\perp B$ to indicate that they are orthogonal with respect to $f$ and denote by $A \oplus B$ their orthogonal direct sum. 
Let $\mathcal L$ be the set of 3-dimensional subspaces $W$ of $V$ of signature (2,1),
i.e. $W=W_{+}\oplus W_{-}$ with $\dim W_{-}=1$ and $\dim W_{+}=2$, and $f$ is positive (resp. negative) definite on $W_+$ (resp. on $W_-$). Define an
equivalence relation $\sim$ on $\mathcal L$ by taking the reflexive and transitive closure of the relation $W_{1}\perp W_{2}$. 

\begin{proposition}\label{sim}
The relation $\sim$ has a single equivalence class on $\mathcal L$.
\end{proposition}

Denote $Q(v)=f(v,v)$ and $\langle v,u\rangle=\text{span}_K\{v,u\}$ for $v,u\in V$.
We say that a vector $v$ or a subspace $W$ is positive (resp. negative) if $Q$ is positive (resp. negative) definite on $\langle v \rangle$ or $W$.
\begin{lemma} \label{neg}
Let $V'\le V$ be a $5$ dimensional subspace and $v_1,v_2$ two negative vectors in $V'$. There is a negative vector $x \in V'$ such that $\langle x,v_1 \rangle$ and $\langle x,v_2 \rangle$ are both non-degenerate and isotropic.
\end{lemma}

\begin{proof}
By multiplying $v_2$ by $-1$ if necessary, we may suppose that 
$$
 Q(v_{1}+v_{2})<\min\left\{  Q(v_{1}),Q(v_{2})\right\}<0.
$$
Since $K=\mathbb{Q}[\sqrt{d}]$ is dense in $\mathbb{R}$ and $\langle v_1,v_2\rangle^\perp \cap V'$ is at least $3$-dimensional and so cannot be negative, we may
find a vector $s\in V'$ orthogonal to $\langle v_1,v_2\rangle$ such that 
$$
 -Q(v_{1}+v_{2})>Q(s)>\max\left\{  -Q(v_{1}%
 ),-Q(v_{2})\right\} >0. 
$$ 
Set $x=v_{1}+v_{2}+s$.
It is easy to see that $x$ is negative and that
$s+v_{2},s+v_{1}$ are positive and lie in $\langle x,v_1 \rangle$ and $\langle x,v_2 \rangle$ respectively. 
\end{proof}

\medskip

\begin{proof}[Proof of Proposition \ref{sim}]
Let $W_1,W_2\in\mathcal L$. Note that $(W_1+W_2)^\perp$, having dimension at least $3$ and rank at most $1$, contains a two dimensional positive subspace $U$.
\medskip

\textbf{Case 1:} Suppose first that $W_{1}+W_{2}$ has rank 1. Then $(W_1+W_2)^\perp$ also has rank one and since its dimension is
at least $3$ it contains a subspace $W_3 \in \mathcal L$. Therefore $W_1 \sim W_3 \sim W_2$.
\bigskip
\medskip

\textbf{Case 2:} Suppose that $W_1$ and $W_2$ share a common negative vector $v$. Then taking $W_4= \langle v \rangle \oplus U \in \mathcal L$ we have that $W_1+W_4$ and $W_2+W_4$ have rank 1 and therefore
$W_1 \sim W_4 \sim W_2$ by Case 1.
\medskip

\textbf{Case 3:} Now suppose that $W_1$ and $W_2$ share a common 2-dimensional positive subspace $Y$. We can write $W_1=Y \oplus \langle v_1 \rangle$, $W_2=Y \oplus \langle v_2 \rangle$ where the negative vectors $v_1,v_2$ are in $Y^\perp$. By Lemma \ref{neg} (applied to a subspace of the space $Y^\perp$) we can find a negative vector $x \in Y^\perp$ such that both $\langle v_1,x \rangle $ and $\langle v_2,x \rangle $ are of signature $(1,1)$.
Put $W_5= Y \oplus \langle x \rangle$. We have that $W_1+W_5= Y \oplus \langle v_1,x \rangle$ and $W_2+W_5= Y \oplus \langle v_2,x \rangle$ are both of rank 1 
and therefore $W_1 \sim W_5 \sim W_2$ by Case 1. \medskip

\textbf{The general case:} We may always choose a positive two dimensional subspace $Y \leq W_1$ and a negative 1-dimensional subspace $Z \leq W_2$. We have $Y+Z \in \mathcal L$ and
$W_1 \sim Y+Z \sim W_2$ by Cases 3 and 2. This completes the proof of Proposition \ref{sim}. 
\end{proof}

\bigskip

For an integral lattice $M$ in the 
$K$-vector space $V$ (i.e. a free $\mathcal O_K$-submodule of $V$ of maximal rank) let us write $GL(M)$ for the arithmetic subgroup of $GL(V)$ which
stabilizes $M$. The following is elementary:

\begin{proposition} \label{com} Let $M_{1},M_{2}$ be two integral lattices in $V$. The groups
$GL(M_{1})$ and $GL(M_{2})$ are commensurable.
\end{proposition}
In fact more is true: The class of arithmetic subgroups is preserved under any rational isomorphism of their ambient algebraic groups, (see \cite{PR}, Propostion 4.1). The next result we need is the following (well known) "arithmetic" variant of the Borel density theorem:

\begin{theorem}\label{borel} Let $k$ be a number field and let $G$ be a $k$-defined simple algebraic group of non-compact type, i.e. such that $G(k_v)$ is noncompact for at least one archimedian valuation $v$ of $k$. Then any arithmetic subgroup $\Gamma$ of $G$ is Zariski-dense in $G$. 
\end{theorem}

To prove Theorem \ref{borel} one can argue as follow. It can be easily deduced from Proposition \ref{com} that the commensurator of $\Gamma$ in $G_{k_v}$ is dense with respect to the local topology. In addition, commensurable algebraic groups must share a common identity connected component.
Therefore the Zariski connected component of the Zariski closure of $\Gamma$ is normal in the simple algebraic group $G$, and hence is either finite or equals $G$. However, by \cite[Corollary 2, page 210]{PR} $\Gamma$ is infinite, and hence it is Zariski dense. An immediate consequence is that $\Gamma$ admits elements of infinite order.

\medskip

For any subspace $W \in \mathcal L$ choose a basis $B_W$ containing a basis for $W$ and a basis for $W^\perp$ and let $M_W$ be the integral lattice it generates. Let $G_W$ be the subgroup of $G$ which stabilizes $W$ and fixes $W^\perp$ pointwise. We have that $G_W$ is a $K$-form of $SO(2,1)$ of non-compact type. The group $\Delta=G_W \cap GL(M_W)$ is an arithmetic subgroup of $G_W$ hence admits a non-torsion element, say $\gamma$. By Proposition \ref{com}, for some $n\in\mathbb{N}$, the nontrivial element $h_W=\gamma^n$ lies in $\Gamma \cap G_W$.  If $W_1,W_2 \in \mathcal L$ and $W_1 \perp W_2$ then by definition the element $h_{W_1}$ commutes with any element of $G_{W_2}$ in particular with $h_{W_2}$. 

We also invoke the celebrated:

\begin{theorem}[Margulis normal subgroup theorem \cite{Mar1}, Ch. VIII]\label{MNS}
A lattice in a higher rank simple Lie group has the normal subgroup property, i.e. all infinite normal subgroups are of finite index.
\end{theorem}

We are now in a position to complete the: 
\bigskip 

\begin{proof}[Proof of Theorem \ref{O}]
Choose any subspace $W \in \mathcal L$ and let $a=h_W \in \Gamma$ be the element defined above. Let $\Gamma_{0}$ be the normal closure of $a$ in $\Gamma.$ By Theorem \ref{MNS}, $\Gamma_0$ is of finite index in $\Gamma$. In particular it is finitely generated, and hence we can find $n$ distinct conjugates $a_{i}=a^{\gamma_i}$ of $a$ ($i=1,\ldots, n$ with $\gamma_i \in \Gamma$) which together generate $\Gamma_0$. 
By Proposition \ref{sim} the sequence $\gamma_1W,\ldots,\gamma_nW$ can be refined to a sequence $W_1,\ldots,W_m\in\mathcal L,~m\ge n$, so that $W_j \perp W_{j+1}$ for $j=1,2, \ldots, m-1$. 
If $W_j=\gamma_iW$ for some $i=1,\ldots,n$, set $b_j=a^{\gamma_i}$, and otherwise set $b_j=h_{W_j}$.
Then the elements $b_i,~i=1,\ldots m$ form an ordered set of generators to $\Gamma_0$ with the desired property.
\end{proof}


\subsection{Cocompact right angled lattices in $SL(n,\mathbb R)$}\label{section:SL(n)}

Let $K$ be a totally real number field of degree $s>1$ over $\mathbb Q$ and let $L$ be a quadratic extension of $K$. Let $\sigma$ be the involution automorphism of Gal$(L/K)$. Suppose that $L$ has $2$ real embeddings $i_1,i_2=i_1 \circ \sigma$ and $2s-2$ complex embeddings $j_1,\bar j_1=j_1 \circ \sigma, \cdots j_{s-1}, \bar j_{s-1}=j_{s-1} \circ \sigma$ into $\mathbb C$. Note that the restrictions of $i_1,j_1, \ldots, j_{s-1}$ to $K$ are all the embeddings of $K$ into $\mathbb R$.
Let $\mathcal O$ be the ring of integers of $K$ and $\mathcal{V}$ the ring of integers of $L$.

For example we can take $K=\mathbb Q(\sqrt{2})$, $\mathcal O= \mathbb Z[ \sqrt{2}]$ and $L$ to be the extension $K(\alpha)$ where $\alpha^2= \sqrt{2}$.

 For $n \geq 3$ let $V=L^n$ and let $f: V \rightarrow K$ be the unitary form defined by
\[ f(x_1, \ldots, x_n)= \sum_{i=1}^n x_i x_i^\sigma\] for any $(x_1, \ldots, x_n) \in V=L^n$.

For a subring $R$ of $L$ such that $R ^\sigma= R$ let us define 
\[ H(R)= \{ x \in SL_n(R) \ | \ x (x^\sigma)^t=\mathrm{Id}\} \]

We are interested in the group $H(L)= SU(f,L)$ and its integral points $H(\mathcal V)$. Note however that the equation $x(x^\sigma)^t= Id$ does not define an algebraic group over $L$ because $\sigma$ is not a linear map over $L$. Following \cite[Section 2.3.3]{PR} we realise $H(L)$ as the $K$-rational points of a related $K$-defined algebraic group $\mathcal G$ using the restriction of scalars of $L$.

Choose a basis $u_1,u_2$ for $L$ over $K$ such that $\mathcal Ou_1 + \mathcal O u_2 \subset \mathcal V$. The left regular representation of $\rho: L \rightarrow End_K(L)$ embeds it as a $K$-subalgebra $\rho(L)$ of $M_2(K)$. This is just the restriction of scalars $R_{L/K}(L)$. Moreover $\rho(L) \cap M_2(\mathcal O)$ is isomorphic to $\mathcal Ou_1 + \mathcal O u_2$.
  
Consider the $K$-algebra $M_n(\rho(L))$ as the subalgebra of $M_{2n}(K)$ consisting of the $n \times n$ matrices with entries in $\rho(L) \subset M_2(K)$. The involution $\sigma$ acts as a $K$-linear automorhism of $M_n(\rho(L))$ by acting on the matrix entries from $\rho(L)$.

Let $\mathbb{G}$ be the $K$-defined algebraic group 
$$
 SU(n,f)= \{ X \in M_n(\rho(L)) \subset M_{2n}(K) | \  X (X^\sigma)^t= Id , \det X=1 \},
$$ 
  
Observe that $\mathbb G (K)= H(L)$. Let $\Gamma= \mathbb{G}(\mathcal O)$ be the corresponding arithmetic subgroup of integral points, and note that $\mathbb G (\mathcal O)$ embeds as a subgroup of finite index in $H(\mathcal V)$.

For each $c=1,\ldots, s-1$ the embedding $j_c : L \rightarrow \mathbb C$ extends to an injection $J_c: \mathbb{G}(K) \rightarrow  SL(n, \mathbb C)$ under which $f$ becomes a positive definite unitary form and therefore $J_c$ induces a $K$-defined homomorphism $J_c: \mathbb{G}(K) \rightarrow SU(n)$.
On the other hand the map $I: \mathbb{G}(K) \rightarrow SL(n,\mathbb R) \times SL(n,\mathbb R)$ given by $ g \mapsto (i_1(g),i_2(g))$ has image $\{(X,(X^{t})^{-1} ) \ | \ X \in SL(n, K) \}$ which is isomorphic to $SL(n, K)$.

By the Borel--Harish-Chandra theorem, the image of the embedding of $\Gamma$ into $D:=SU(n)^{s-1} \times SL(n,\mathbb R)$ given by $g \mapsto (J_1(g), \ldots, J_{s-1}(g), I(g))$ is a lattice in $D$ and since the first $s-1$ factors are compact, it projects to a lattice in $SL(n,\mathbb R)$.
Moreover $\Gamma$ is cocompact because the form $f$ is anisotropic, i.e. $f(v)=0$ for $v \in V$ implies $v=0$. This can be seen from the fact that under the embedding $J_1: \mathbb{G}(K) \rightarrow SU(n)$, $f$ becomes a positive definite form.

\begin{lemma} 
There exists an element $u$ which is a unit of infinite order in $\mathcal V^*$ and such that $uu^\sigma=1$.
\end{lemma}

\proof The map $ x \mapsto xx^\sigma$ is a homomorphism $\pi$ from $\mathcal V^*$ to $\mathcal O^*$. By the Dirichlet unit theorem $\mathcal O^*$ has torsion free rank $s-1$ while $\mathcal V^*$ has rank $s$. Therefore $\ker \pi$ contains some element of infinite order in $\mathcal V^*$.
$\square$

\begin{theorem} \label{unit} The group $\Gamma$ is virtually right-angled.
\end{theorem}

Let $B= \{e_1, \ldots, e_n\}$ be any $f$-orthogonal basis of $V$ over $L$, i.e. $f(e_i,e_j)=0$ for all $i \not = j$. Define the following diagonal subgroup of $H(L)$ with respect to the basis $B$: 

\[
 A_B= \{ diag(u^{d_1},u^{d_2}, \ldots ,u^{d_n}) \ | \ d_i \in \mathbb Z, \ d_1 + \cdots + d_n=0 \}.
\]

Note that, under the identification of $H(L)$ with $\mathbb{G}(K)$ we have that $A_B < \mathbb{G}(K)$ since $u u^\sigma=1$. The group $A_B$ preserves the $\mathcal V$-integral lattice of $V$ with respect to the basis $B$. In view of Proposition \ref{com}, $|A_B: A_B \cap H(\mathcal V)|$ is finite and therefore the group $T_B:=A_B\cap \Gamma$ is of finite index in $A_B$.

Define a relation on the set $X$ of orthogonal bases of $V$ over $L$ by declaring $B_1 \sim B_2$ if $B_1$ and $B_2$ share a vector.

Suppose that this is the case, specifically $B_1=\{ e_1,e_2, \cdots e_n\}$, $B_2= \{e_1,e_2', \cdots , e_n'\}$. Then $Le_2+ \cdots +Le_n= (Le_1)^{\perp}= Le'_2+ \cdots +Le'_n$ and so $A_{B_1} \cap A_{B_2}$ contains the infinite cyclic subgroup

\[ \{ diag(u^{-nd+d},u^{d}, \ldots, u^{d}) \ | \ d \in \mathbb Z \}\]
 (the definition gives the same group with respect to either basis $B_1$ or basis $B_2$)
 We can deduce that $A_{B_1} \cap A_{B_2}$ is infinite and hence $T_{B_1} \cap T_{B_2}$ is infinite.

Finally we need to check

\begin{lemma} \label{close} The transitive closure of $\sim$ is all of $X$, all orthonormal bases.
\end{lemma} 

\begin{proof} 
Consider two arbitrary elements $B_1,B_2\in X$.
Let $e_1 \in B_1$ and $e_1' \in B_2$.

If $e_1$ and $e_1'$ are scalar multiples of each other we just rescale one to the other without changing the rest. Otherwise the subspace $W=Le_1+Le_2$ is two dimensional, in which case take an orthogonal basis  $C$ for $W^\perp$. Choose a vector $h \in W$ orthogonal to $e_1$ and a vector $d \in W$ orthogonal to $e_1'$. Now 
\[ B_1 \sim \{h,e_1\} \cup C \sim \{d,e'_1\} \cup C \sim B_2. \] 
\end{proof}

We can now complete the proof of Theorem \ref{unit}.
Let $B=B_1$ be any orthonormal basis for $V$ and choose $g_1 \in T_{B_1} \backslash \{1\}$. The normal closure of $g_1$ in $\Gamma$ has finite index in $\Gamma$ by Theorem \ref{MNS}. Therefore there exist congugates $g_2, \ldots,g_k \in g^\Gamma$ such that $\langle g_1 \ldots, g_k \rangle$ has finite index in $\Gamma$. Now each $g_i$ lies in some $T_{B_i}$ for an appropriate basis $B_i \in X$. By Lemma \ref{close} we can find some sequence of bases $C_1, \ldots, C_m$ in $X$ which contains all $B_j$ and such that $C_i \sim C_{i+1}$ for each $i=1, \ldots, m-1$. This means that $T_{C_1}, \ldots, T_{C_n}$ is a sequence of abelian groups with infinite consecutive intersections whose union contains all of $g_j$ and therefore generates a subgroup of finite index in $\Gamma$. 
\qed


\section{Non-archimedean local fields} \label{sec:local}

Recently many of the results of \cite{Samurais} have been extended by A. Levit and the second named author to IRS in simple Lie groups over non-archimedean local fields \cite{GL} (this work is under preparation and is expected to be on the Arxiv soon). Below we shall use them in combination with Theorem \ref{ra} to prove Theorem \ref{p}.

We shall make use of the following definition from \cite{GL}:

\begin{definition}
Let $G$ be a locally compact group. A family of IRS, $\mathcal{F}\subset\text{IRS}(G)$ is {\it weakly uniformly discrete} if for every $\epsilon>0$ there is an identity neighbourhood $U\subset G$ such that for all $\mu\in\mathcal{F}$ we have
$$
 \mu \{\Gamma:\Gamma\cap U\ne\text{id}_G)\}<\epsilon.
$$
\end{definition}

One advantage of this notion of uniform discreteness is that it also applies to non-uniform lattices. In particular, the following proposition is straightforward:

\begin{proposition}\label{prop:WUD}
Let $\Gamma\le G$ be a lattice and let $\mathcal{F}$ be the family consisting of all $\mu_\Delta$ where $\Delta$ runs over the finite index subgroups of $\Gamma$. Then $\mathcal{F}$ is weakly uniformly discrete.\qed
\end{proposition}

The following theorem is proved in \cite{GL}:

\begin{theorem}\label{thm:LF}
Let $k$ be a non-Archimedean local field (of positive characteristic) and let $G=\mathbb{G}(k)$ be the group of $k$ rational points of a simple algebraic group $\mathbb{G}$ of $k$-rank at least $2$. Let $\mu_n$ be a weakly uniformly discrete sequence of distinct non-atomic IRS in $G$. Then $\mu_n\to \mu_{\text{id}_G}$. 
\end{theorem}

Note that Theorem \ref{thm:LF} is an analogue of Theorem \ref{mu_n} for higher rank simple groups over general local fields. The $0$ characteristic (non-Archimedean) case is more general and easier (see \cite{GL}). The proof of Theorem \ref{thm:LF} makes use of ideas originated in \cite{Samurais} and 
relies on a version of the Nevo--Stuck--Zimmer theorem which was recently obtained by Arie Levit (see \cite{Arie}). 

The following is an immidiate consequence of Lemma \ref{induced}, Theorem \ref{thm:LF} and Proposition \ref{prop:WUD}:

\begin{corollary}
Let $G$ and $\Gamma$ be as in Theorem \ref{thm:LF}, and let $\Gamma_n\le \Gamma$ be a sequence of subgroups of finite index, such that $|\Gamma:\Gamma_n|\to\infty$. Then $(\Gamma_n)$ is Farber.
\end{corollary}

Theorem \ref{p} now follows from Theorem \ref{thm:LF} and Theorem \ref{ra} by observing that  $\Gamma$ is generated by its elementary subgroups $E_{i,j}$ for $1 \leq i \not = j \leq n$ which can be enumerated as $\{A_{n}\}_{n=1}^{n^2-n}$ such that $A_k$ commutes with $A_{k+1}$ for every $1 \leq k \leq n^2-2$.

\begin{remark}
It is shown in \cite{WUD} that if $G$ is a semisimple Lie group over a local field of characteristic $0$, then the set of all IRS in $G$ which have no atom at $G$ is weakly uniformly discrete.
The analog statement for local field of positive characteristic is still unknown.
\end{remark}

\section{Open problems} \label{sec:open}

In this section we present some open problems and suggest further directions
for research. \medskip

\subsection{The growth of rank for lattices in a fixed Lie group.} We
already pointed out in the introduction that Conjecture \ref{kerdes} does not hold for
rank $1$ simple Lie groups in general. In particular, lattices surjecting to a
non-Abelian free group provide easy counterexamples. However, all these give
sequences of lattices that do not approximate the ambient Lie group in the
following sense.

\begin{definition}
Let $G$ be a Lie group. A sequence of lattices $\Gamma_{n}$ is Farber, if the
invariant random subgroups $\mu_{\Gamma_{n}}\rightarrow\mu_{1}$.
\end{definition}

We suggest that the asymptotic growth of the rank of a lattice only depends on its
covolume and ambient Lie group.

\begin{problem}
\label{bigproblem}Let $G$ be a connected Lie group and let $\Gamma_{n}$ be a
Farber sequence of lattices in $G$. Then
\[
\frac{d(\Gamma_{n})-1}{\mathrm{Vol}(\Gamma_{n})}%
\]
converges.
\end{problem}

We further suspect that the above limit (which only depends on $G$) should be
$0$ except when $G$ satisfies the following:

\begin{itemize}
\item The amenable radical $A\lhd G$ is compact;

\item $G/A\cong\mathrm{PSL}(2,\mathbb{R})$.
\end{itemize}

Note that by \cite{Ge} there is a constant $C=C(G)$ such that
\[
d(\Gamma)\leq C\cdot\mathrm{Vol}(\Gamma)
\]
for every lattice $\Gamma\leq G$. In particular, this implies that the limsup
of the sequence in Problem \ref{bigproblem} exists and is at most $C$.

The corresponding question about discrete groups is as follows.

\begin{problem}
\label{discrank}Let $\Gamma$ be a finitely generated group and let $\Gamma
_{n}$ be a Farber sequence of subgroups of finite index in $\Gamma$. Then
\[
\frac{d(\Gamma_{n})-1}{\left\vert \Gamma:\Gamma_{n}\right\vert }%
\]
converges.
\end{problem}

In \cite{miknik} the weaker question is investigated: is the rank gradient
of a Farber chain  independent of the chain? That question is related to the
fixed price problem of Gaboriau, that is, whether every free action of a
countable group must have the same cost. In particular, when the ambient group
has fixed price, the rank gradient is independent of the chain (assuming it is
Farber). To the best of our knowledge a cost correspondence of this nature
has not been worked out in the realm of Lie groups. \medskip

\subsection{The distortion function of a p.m.p. action.} Let $\Gamma$ be
generated by the finite symmetric set $S$ and let $\Phi$ be a p.m.p. action of
$\Gamma$ on the Borel space $(X,\mu)$. One can define the bi-Lipshitz
distortion function of $\Phi$ as follows. Let $\mathcal{G}$ be the measured
groupoid associated to the action. Let the standard generating set
\[
R_{S}=\left\{  (x,s)\in\mathcal{G}^{1}\mid x\in X,s\in S\right\}
\]
and let
\[
N=\left\{  (x,e)\in\mathcal{G}^{1}\mid x\in X\right\}  \text{.}%
\]

Recall that the cost of $\mathcal{G}$ is defined as
\[
\mathrm{c}(\mathcal{G})=\inf_{\substack{T\subseteq\mathcal{G}^{1}\text{
Borel}\\T\text{ generates }\mathcal{G}}}\widetilde{\mu}(T)
\]
Recall that in Definition \ref{groupoiddist} we defined the distance of two generating sets $A,B\subseteq\mathcal{G}^{1}$ as
\[
d_{L}(A,B)=\inf\left\{  k\in\mathbb{N}\mid A\subseteq(B\cup B^{-1}\cup
N)^{k}\text{ and }B\subseteq(A\cup A^{-1}\cup N)^{k}\right\}  \text{.}%
\]

A priori, the distance of two generating sets can be infinite. However, using
a compactness argument, it is proved in \cite{miknik} that for any
$\varepsilon>0$ there exists a generating set $T\subseteq\mathcal{G}^{1}$ with
$\widetilde{\mu}(T)\leq\mathrm{c}(\mathcal{G})+\varepsilon$ such that
$d_{L}(T,R_{S})$ is finite.

\begin{definition}
Let the distortion function $\delta:\mathbb{R}^{+}\rightarrow\mathbb{R}$ of
$\Phi$ be defined as
\[
\delta(x)=\inf_{\substack{T\text{ generates }\mathcal{G}\\\widetilde{\mu
}(T)\leq\mathrm{c}(\mathcal{G})+1/x}}d_{L}(T,R_{S})\text{.}%
\]

\end{definition}

The distortion function $\delta$ does depend on $S$, however, it is easy to
see that it is well-defined up to a scale factor over all finite generating
sets of $\Gamma$.

One can show the following.

\begin{theorem}
\label{radistortion}Let $\Gamma$ be a right angled group and let $\Phi$ be a
free p.m.p. action of $\Gamma$. Then $\Phi$ has polynomial distortion.
\end{theorem}

The proof is implicitly contained in Gaboriau \cite{gabor}. It follows the same
argument as in the proof of Theorem \ref{ccsofapprox} except that one may not be able to choose a maximal
$R$-separated vertex set for the $C_{i}$-cycle metric in a measurable way.
However, it is possible to do this by paying an arbitrary small extra price
using Rokhlin's lemma. Indeed, the powers of each generator $g_{m}$ provide a free action of
$\mathbb{Z}$, and thus one can choose a measurable subset $X_{m}\subseteq X$
of measure $1/R+\varepsilon$ such that
\[
\bigcup\limits_{i=0}^{R-1}X_{m}g_{m}^{i}=X\text{.}%
\]

The above distortion function seems to be interesting already for amenable
groups. By the Ornstein-Weiss theorem \cite{ornweiss}, all free actions of
such groups are orbit equivalent to $\mathbb{Z}$, and thus have cost $1$, but
the orbit equivalence may be a quite complicated coding and the theorem does
not tell anything about its distortion. Torsion free nilpotent groups are
right angled, so they fall under Theorem \ref{radistortion}, but already for
solvable groups, we do not see any general upper bound for the distortion function. In particular, there
may be examples of exponential distortion already in this class. For arbitrary
groups, we do not know the answer to the following.

\begin{problem}
Is the distortion function of a free p.m.p. action of a finitely generated
group always bounded by an exponential function?
\end{problem}

Note that by adapting the Abert-Weiss theorem on weak containment
\cite{abeweiss}, one can show that the Bernoulli actions of $\Gamma$ dominate
all free actions of $\Gamma$ in terms of the distortion function. \medskip

\subsection{The growth of the first homology torsion.} The first problem
listed here (communicated to us by Marc Lackenby) clearly shows the general
lack of understanding on how to estimate the torsion homology from below.

\begin{problem}
\label{Problem:torsion} Is there a finitely presented group $\Gamma$ and a
normal chain $\Gamma_{n}$ in $\Gamma$ with trivial intersection with
\[
\mathrm{Trs}(\Gamma,(\Gamma_{n}))>0\text{?}%
\]

\end{problem}

The expected results are much bolder. In particular, it is conjectured by
Bergeron and Venkatesh \cite{berg} that for any closed arithmetic hyperbolic 3-manifold group $\Gamma$ and congruence chain
$(\Gamma_{n})$ in $\Gamma$, one has
\[
\mathrm{Trs}(\Gamma,(\Gamma_{n}))=\frac{1}{6\pi}\mathrm{Vol}(\Gamma)\text{.}%
\]
This, in particular, would be a rather direct way to express the covolume of
$\Gamma$ using torsion homology. The later paper of Bergeron, Sengun and Venkatesh \cite{bergsengvenk} advances towards a $3$-manifold group example for Problem \ref{Problem:torsion} but it is still conditional on a conjectured density property of a certain spectral measure. \medskip

A very recent result of Bader, Gelander and Sauer
\cite{BGS} states that for every $d\neq3$ there is a constant $C=C(d)$ such
that for every negatively curved analytic $d$-manifold $M$ with sectional
curvature bounded below by $-1$, we have
\[
\log\mathrm{Trs}\mathrm{H}_{k}(M,\mathbb{Z})\leq C\cdot\mathrm{Vol}(M)
\]
for all $k$. In particular if $\Gamma$ is a lattice in a rank one Lie group
and $\Gamma_{n}\leq\Gamma$ is as in Problem \ref{Problem:torsion} then
$\mathrm{limsup}\log\mathrm{Trs}\mathrm{H}_{1}(\Gamma_{n},\mathbb{Z})$ is finite.
Conjecturally, the analogue of the result of \cite{BGS} should hold for every
simple Lie group not locally isomorphic to $\mathrm{PSL}(2,\mathbb{C})$. 

One could attempt to extend the Bergeron--Venkatesh conjecture to a general
problem in the spirit of Problem \ref{bigproblem}, using normalized log
torsion instead of the normalized rank. However, as a recent example of Brock
and Dunfield \cite{dubro} shows, already for hyperbolic $3$-manifolds, this
direct translation would fail. They prove the existence of a Farber sequence of
hyperbolic integer homology $3$-spheres.
This in particular means that the first
homology torsion stays trivial. Note that as Brock and Dunfield \cite[1.11]%
{dubro} point out, this apparent deviation from the Bergeron-Venkatesh
conjecture may be because the right invariant to study in this situation also
has to involve the so-called regulator. 
In the other direction, it is shown in the recent preprint by Bader, Gelander and Sauer, \cite{BGS} that for any function $f:\mathbb{N}\to\mathbb{R}^+$ there is a Farber sequence of (non-arithmetic) compact rational homology spheres hyperbolic $3$ manifolds
whose normalised torsion grows faster than $f$. By the main result of \cite{BGS} this can happen only in dimension $3$.
The conjecture \cite[Conjecture
1.13]{dubro} stated for general Lie groups instead of just $\mathrm{SL}%
(2,\mathbb{C})$, suggests itself as the right analogue of Problem \ref{bigproblem}
for torsion. \medskip\

\subsection{Higher degree torsion}

It is natural to ask whether the right angled condition implies the vanishing
of the torsion growth on higher homologies. As we expect the second homology
torsion growth of $SL(3,\mathbb{Z})$ to be positive, the direct answer is no.
However, note that in a work in progress, Abert,
Bergeron and Gaboriau introduce a higher homotopy version of being right
angled and prove the vanishing of the torsion growth on some higher homologies
for these groups. Note that as of now, their approach do not apply to cocompact
lattices. Also note that the congruence subgroup property does not seem to be
relevant for higher torsion. So, as of now, we do not know a cocompact lattice
with vanishing second torsion homology growth over Farber sequences.

\end{document}